\documentclass[12pt,centertags,oneside]{amsart}

\RequirePackage[utf8]{inputenc}

\usepackage{amstext,amsthm,amscd,accents}
\usepackage{amsmath}
\usepackage{amssymb}
\usepackage[commandnameprefix=ifneeded]{changes}
\usepackage{enumitem}
\usepackage[mathscr]{eucal}
\usepackage{mathrsfs}
\usepackage{charter}
\usepackage{xcolor}
\usepackage[a4paper,width=16.2cm,top=3cm,bottom=3cm]{geometry}
\usepackage{appendix}
\usepackage{multicol}

\definecolor{paperblue}{RGB}{0,102,204}
\definecolor{paperred}{RGB}{180,30,30}
\definecolor{papergreen}{RGB}{0,130,70}

\usepackage[
  colorlinks=true,
  linkcolor=paperblue,
  citecolor=paperblue,
  urlcolor=paperblue
]{hyperref}
\usepackage[capitalize,nameinlink]{cleveref}

\numberwithin{equation}{section}

\newtheorem{theorem}{Theorem}[section]
\newtheorem{definition}[theorem]{Definition}
\newtheorem{proposition}[theorem]{Proposition}
\newtheorem{corollary}[theorem]{Corollary}
\newtheorem{lemma}[theorem]{Lemma}
\newtheorem{remark}[theorem]{Remark}
\newtheorem{example}[theorem]{Example}
\newtheorem*{assumption}{Main Assumption (MA)}

\crefname{theorem}{Theorem}{Theorems}
\crefname{definition}{Definition}{Definitions}
\crefname{proposition}{Proposition}{Propositions}
\crefname{corollary}{Corollary}{Corollaries}
\crefname{lemma}{Lemma}{Lemmas}
\crefname{remark}{Remark}{Remarks}
\crefname{example}{Example}{Examples}
\crefname{equation}{Equation}{Equations}

\newcommand{\cali}[1]{\mathscr{#1}}

\newcommand{\supp}{{\rm supp}}
\newcommand{\dist}{\mathop{\mathrm{dist}}\nolimits}

\newcommand{\ddc}{{\rm dd^c}}
\renewcommand{\d}{{\rm d}}
\newcommand{\nuinfty}{\nu^{(\infty)}}
\newcommand{\Sinfty}{S^{(\infty)}}

\newcommand{\PSH}{{\rm PSH}}
\newcommand{\DSH}{{\rm DSH}}
\newcommand{\id}{{\rm id}}

\newcommand{\vep}{\varepsilon}
\newcommand{\massump}{\hyperref[assumption]{(MA)}}
\newcommand{\shad}{{\rm shad}}
\newcommand{\Gr}{{\rm Gr}}
\newcommand{\Gl}{{\rm GL}}

\newcommand{\Sc}{\cali{S}}
\newcommand{\B}{\mathbb{B}}
\newcommand{\C}{\mathbb{C}}
\newcommand{\G}{\mathbb{G}}
\newcommand{\bbT}{\mathbb{T}}

\newcommand{\N}{\mathbb{N}}

\newcommand{\R}{\mathbb{R}}
\newcommand{\lp}{\langle}
\newcommand{\rp}{\rangle}
\renewcommand\P{\mathbb{P}}

\makeatletter
\newcommand{\wwhat}[2][0mu]{%
  \mathchoice
    {\wwhat@{\displaystyle}{#2}{#1}}%
    {\wwhat@{\textstyle}{#2}{#1}}%
    {\wwhat@{\scriptstyle}{#2}{#1}}%
    {\wwhat@{\scriptscriptstyle}{#2}{#1}}%
}
\newcommand{\wwhat@}[3]{%
  \ooalign{%
    $#1\widehat{#2}$\cr
    \hidewidth
    \raisebox{0.35ex}{$#1\mkern#3\widehat{\phantom{#2}}$}%
    \hidewidth\cr
  }%
}
\makeatother

\newcommand{\hhatf}{\rule{0pt}{2.4ex}\wwhat[5.5mu]{f}}
\newcommand{\hhatT}{\rule{0pt}{2.4ex}\wwhat[3mu]{T}}

\newcommand{\blue}[1]{{\color{paperblue}#1}}
\newcommand{\red}[1]{{\color{paperred}#1}}

\title{Equidistribution of saddle periodic points for H\'enon-like maps}

\author{Muhan Luo}\address{Department of Mathematics,  National University of Singapore - 10, Lower Kent Ridge Road - Singapore 119076}\email{e0708207@u.nus.edu}

\author{Qi Zhou}\address{Department of Mathematics,  National University of Singapore - 10, Lower Kent Ridge Road - Singapore 119076}\email{e1124859@u.nus.edu}

\begin{document}

\begin{abstract}
    We prove that under a natural assumption on the dynamical degrees, the saddle periodic points of a H\'enon-like map in any dimension equidistribute with respect to the equilibrium measure. Our work is a generalization of the results of Bedford-Lyubich-Smillie, Dujardin, and Dinh-Sibony along with improvements of their techniques. We also investigate some fine properties of Green currents associated with the map. 
\end{abstract}


\maketitle
\setcounter{tocdepth}{1}
\tableofcontents

\section{Introduction}
The systematic study of the dynamics of (generalized) H\'enon maps was started by Bedford, Lyubich, and Smillie. Using the tools from complex geometry and pluripotential theory, they proved in a series of papers \cite{BLS93a,BLS93b,BS91,BS92} that these maps enjoy many good dynamical properties. In particular, there exists a unique measure of maximal entropy which is mixing, hyperbolic, and the periodic points equidistribute with respect to this measure. 

Generalized H\'enon maps provide rich classes of examples of non-uniformly hyperbolic maps. In order to study a wider range of dynamical systems with some hyperbolicity, Sibony introduced in \cite{S99PS} the notion of {\it regular automorphisms} (or {\it H\'enon-type maps}). These are polynomial automorphisms on $\C^k$ whose indeterminacy sets on $\P^k$ are disjoint. With this simple condition, most of the properties of generalized H\'enon maps extend successfully to regular automorphisms. See for example \cite{BD24FMS,dTV10MSMF,DS09AM,DS16,VW24Ax} and the survey \cite{FS94NATOASI}. 

Instead of restricting ourselves to polynomial maps, we can also consider holomorphic maps with some weak hyperbolicity in a geometric sense, i.e., some rough contraction and expansion in different directions. Pioneering work was done by Dujardin \cite{Dujardin04} where he considered {\it H\'enon-like maps} on $\C^2$, which were first introduced by Hubbard and Oberste-Vorth in \cite{HO}. He proved that many of the properties of generalized H\'enon maps are still true in this case. Later, Dinh-Sibony \cite{DS06AIF} initiated the study of H\'enon-like maps in any dimension. They showed the existence of an equilibrium measure $\mu$ which is mixing and of maximal entropy. A subsequent paper of Dinh-Nguyen-Sibony \cite{DNS} proved that $\mu$ is exponential mixing and hyperbolic. Recall that an invariant positive measure is \textit{hyperbolic} if its Lyapunov exponents are non-zero.

Whether some other properties of H\'enon maps hold for general H\'enon-like maps remains unclear until the recent release of two papers of Bianchi-Dinh-Rakhimov \cite{BDR23,BDR24}. Amongst other results they obtained, they proved that under a natural condition on the dynamical degrees, the Green currents of a H\'enon-like map are woven. See \cref{sec:tame} for relevant definitions. In this paper, we solve another important and long-standing question regarding H\'enon-like maps: the saddle periodic points are equidistributed with respect to the equilibrium measure $\mu$. 

To state our main theorem, we need to introduce some relevant definitions. We will treat them in \cref{sec:revisit} in a more rigorous and systematic way. Let $1\leq p\leq k-1$ be fixed, and let $M\subset \C^p$ and $N\subset \C^{k-p}$ be two bounded convex domains. {Define} $D:=M\times N\subset \C^k$. Then a {\it H\'enon-like map} $f$ is a holomorphic map from a vertical subdomain of $D$ to a horizontal subdomain of $D$. So, in some sense, the map is expanding in $p$ directions and contracting in $k-p$ directions. {It can be seen as a complex and non-uniformly hyperbolic version of the Smale horseshoe. Examples of H\'enon-like maps include restrictions of regular automorphisms of $\C^k$ to a large enough polydisk \cite{DNS}. Maps of the form $(z,w)\mapsto (h(z)+\vep w, \vep' z)$ for a polynomial-like map $h(z)$ are also H\'enon-like on suitable domains, see \cite{Dujardin04}. Small perturbations of H\'enon-like maps are still H\'enon-like (up to slightly shrinking the domain $D$). Therefore, this family is very rich and of infinite dimension.}

Related to {a H\'enon-like map} $f$, we can define two sequences of positive numbers $\{d_l^+\}_{0\leq l\leq p}$ and $\{d_l^-\}_{0\leq l\leq k-p}$ called the {\it dynamical degrees}. They measure the growth rates of masses of positive closed vertical (or horizontal) currents of various bi-dimensions under the backward (or forward) iterations of $f$. In particular, $d:=d_p^+=d_{k-p}^-$ is an integer called the {\it main dynamical degree}. The topological entropy of $f$ and the entropy of $\mu$ are both equal to $\log d$. Throughout {all} this paper, we make the following assumption for a H\'enon-like map $f$:
\begin{assumption}\label{assumption}
    $d>\max (d_{p-1}^+,d_{k-p-1}^-)$.
\end{assumption}
This is equivalent to $d>1$ when $k=2$, i.e., $f$ has positive entropy. {In fact, the condition needed in this paper is that $d$ is strictly larger than all the other dynamical degrees. This follows directly from \massump\ and the main result of \cite{BDR23}. 

The examples mentioned above, including restrictions of regular 
automorphisms of $\mathbb{C}^k$ and suitable perturbations of polynomial-like maps,  all satisfy \massump. It is also easy to see that \massump\ is stable under small perturbations. Therefore, the class of H\'enon-like maps satisfying 
\massump\ is still very rich and of infinite dimension.} Our main theorem is as follows.

\begin{theorem}\label{thm:main}
    Let $f$ be a H\'enon-like map on $D$ satisfying \massump. Let $P_n$ be the set of periodic points of period $n$ of $f$, and let $SP_n\subset P_n$ be the saddle periodic points. Let $Q_n$ be either $P_n$ or $SP_n$. Then {we have}
    \[
        d^{-n}\sum_{z\in Q_n}\delta_z\to \mu
    \]
    as $n$ goes to infinity where $\delta_z$ denotes the Dirac measure at $z$.
\end{theorem}

Here is our strategy of proof. Let $\Gamma_n$ be the graph of $f^n$ in $D\times D$ and $\Delta$ be the diagonal in $D\times D$. The proof is based on the observation that the set of periodic points of period $n$ can be identified with the intersection of $\Gamma_n$ with $\Delta$. On the other hand, we can consider the map $F=(f,f^{-1})$ which is still a H\'enon-like map on $D\times D$. Then $\Gamma_n$ is the pullback of $\Delta$ under some iterations of $F$. When $n$ goes to infinity, $d^{-n}[\Gamma_n]$ converges to the Green current of $F$ whose intersection with $\Delta$ is exactly the measure $\mu$. Therefore, {to prove \cref{thm:main}} it suffices to prove the following equality: 
 \begin{equation}\label{commu-intersec}
    \lim_{n\to\infty}d^{-n}[\Gamma_n]\wedge [\Delta]=(\lim_{n\to\infty}d^{-n}[\Gamma_n]) \wedge [\Delta].    
 \end{equation} 
 In the general setting of currents, this fails.  A reason is that there can be too {many} tangent vectors of $\Gamma_n$ which are close to the tangent vectors of $\Delta$. Using density theory of currents {developed in \cite{DS18JAG} and following the idea in \cite{DS16}}, we quantify this property and show {the tangency between $\Gamma_n$ and $\Delta$} is negligible when $n\to \infty$, {i.e., most of the intersections are transversal.} This is done by lifting $\Gamma_n$ to $\widehat{\Gamma}_n$ in the Grassmannian bundle $\Gr(D\times D,k)$ where the tangent directions to $\Gamma_n$ can be easily observed. Denote by $\Pi:\Gr(D\times D,k)\to D\times D$ the projection map. The intersection of $\widehat{\Gamma}_n$ with $\Pi^{-1}(\Delta)$ tells the directions of $\Gamma_n$ when intersecting with $\Delta$. Therefore, we study the limit currents of $d^{-n}[\widehat{\Gamma}_n]$ and their intersections with $\Pi^{-1}(\Delta)$.
 
 Next, we consider the projection $\pi$ from $D\times D$ to $\C^k$ which maps $(x,y)$ to $z=y-x$. This is the projection along the directions of 
 $\Delta$. We show that this is a ramified covering of degree $d^n$ when restricted on $\Gamma_n$ near the diagonal. The {tangencies} between $\Gamma_n$ and $\Delta$ {reflect} the ramification of $\pi|_{\Gamma_n}$ near the diagonal. Since {they are} negligible, $\pi|_{\Gamma_n}$ is almost a covering of degree $d^n$ near the diagonal when $n$ is large enough. This is enough to finish the proof.
 
 \cref{thm:main} is shown by Dinh-Sibony \cite{DS16} for regular automorphisms. Compared with their work, there are several difficulties that require extra effort to handle. In the case of regular automorphisms, many key estimates are deduced from the action of $f$ over the Hodge cohomology classes. When $f$ is a H\'enon-like map, this is not applicable since we are restricted to an open subset $D\subset\C^k$. {Moreover}, the Green currents of $f$ are no longer wedge products of positive closed $(1,1)$-currents and are not extremal anymore. All these {issues} cause many challenges in the study of H\'enon-like maps in higher dimension, see also \cite{BDR23,BDR24}. We need to develop and apply a different intersection theory. Some other techniques also need to be improved. 

The previous discussion motivates us to study the lifts of submanifolds to the Grassmannian bundle. Let $\Sigma$ be a horizontal submanifold of dimension $p$ in $D$. Let  
\[
    {\widehat{\Sigma}:=\{(x,[v])\in D\times \G(p,k): x\in\Sigma, v \text{ represents }T_x\Sigma\}}
\]
be its canonical lift to $\Gr(D,p)=D\times \G(p,k)$, where $\G(p,k)$ denotes the Grassmannian manifold consisting of all $p$-dimensional linear subspaces of $\C^k$. Let $\widehat{f}$ be the canonical lift of $f$ { to $\Gr(D,p)$ defined by $\widehat{f}(x,v):=(f(x),\d f_x(v))$ for any $x\in D$ and $v\in\G(p,k)$}. Consider the iteration $\widehat{\Sigma}_n:=\widehat{f}^n(\widehat{\Sigma})$. By \cite{DNS,DS06AIF}, $d^{-n}[f^n(\Sigma)]$ converges to a woven positive closed current $T^-$, called the {\it Green current} of $f$. {For the definition of woven currents, see \cref{defi:woven}.} In \cref{sec:tame}, we strengthen this result and prove that any limit of $d^{-n}[\widehat{\Sigma}_n]$ is a lift of $T^-$. The proof uses a result of de Th\'elin \cite{dTh}. Our idea is to consider the {lift of $\widehat{\Sigma}_n$ to the Grassmannian bundle $\Gr(\Gr(D,p),p)$} and bound their masses by $d^n$. It follows that any limit of $d^{-n}[\widehat{\Sigma}_n]$ is a lift of $T^-$ plus a possible positive closed current whose projection to $D$ is 0. {To eliminate the latter one, we apply \massump\ to show that its {\it shadow}, which is a positive closed horizontal current on $D$ of bi-dimension strictly less than $p$, is zero.} 

 For positive closed currents of higher bi-degrees, two intersection theories were established by Dinh-Sibony: the one given by {\it super-potentials} introduced in \cite{DS09AM} as generalizations of local potentials and the other coming from density theory in \cite{DS18JAG}. One could ask whether these two theories agree with each other. On a compact K\"ahler manifold, this is studied and confirmed in \cite{DNV}. For super-potentials defined in our non-compact setting, this is unanswered to the best of our knowledge. This question is of independent interest from the point of view of pluripotential theory. We refer the readers to \cite{DS18JAG,HKV,KV,Nguyen,Nguyen25,Vu21} for the relations between density theory and other intersection theories. This question is also important for our proof since we use these two notions of intersection simultaneously. In \cref{sec:sp}, we give an affirmative answer to this question.

 In \cref{sec:oseledec}, we give a characterization of the {\it unstable Oseledec measure} denoted by $\mu_-$. It is a lift of $\mu$ to the Grassmannian bundle $\Gr(D,p)$. We show that for any positive smooth $(k-p+m,k-p+m)$-form $\widehat{\gamma}$ on $\Gr(D,p)$ whose support {is suitably restricted}, $d^{-n}(\widehat{f}^n)_*(\widehat{\gamma})\wedge\pi_D^*(T^+)$ converges to $\mu_-$ when $n$ goes to infinity. Although a similar result is shown in \cite{DS16} for regular automorphisms, their proof is not applicable to our case. We first reduce the problem to the case where $\widehat{\gamma}$ is of the form $\alpha_0\wedge\Omega_\G$. Here $\alpha_0$ is cohomologous to $T^-$, and $\Omega_\G$ is a volume form on $\G(p,k)$. {Then, letting $\nu$ be a limit} value of the sequence, we prove that its backward iteration $(\widehat{f}^l)^*\nu$ converges to $\mu_-$. Our proof is based on a {combination} of Pesin theory with pluripotential theory. This allows us to show a more general result than the one in \cite{DS16}.

 We complete the proof of our main theorem in \cref{sec:main} following the strategy discussed previously. 
 
 Through all the paper, we focus on the main difficulties and refer the reader to the literature for arguments already available elsewhere.

\subsection*{Acknowledgements} We thank the referees for the many helpful comments and suggestions, which have significantly improved the quality and presentation of this paper. The first author is supported by NUS Research Scholarship and the NUS grant A-8002488-00-00. The second author is supported by President's Graduate Fellowship. We would like to thank our supervisor Tien-Cuong Dinh for many helpful discussions and for carefully reading the draft. We also thank Hoang-Son Tran and Duc-Bao Nguyen for constant support during the preparation of this paper. 

\section{H\'enon-like maps}\label{sec:revisit}

 In this section, we recall some basic definitions and properties related to H\'enon-like maps. Additionally, we prove several results which will be useful later but are not mentioned elsewhere.

 \subsection{Preliminaries on H\'enon-like maps}
 
 Let $M\subset \C^p$ and $N\subset \C^{k-p}$ be two bounded convex open sets. Consider the bounded convex domain $D= M\times N \subset \mathbb C^p \times \mathbb C^{k-p}$. {The \textit{vertical} (resp. \textit{horizontal}) boundary of $D$ is the set $\partial_v D:=\partial M\times \overline{N}$ (resp. $\partial_h D:= \overline{M}\times \partial N$).} We denote by $\pi_1$ and $\pi_2$
the first and the second projections of $D\times D$ on its factors, respectively. 

\begin{definition}\rm\label{defi:HLM}
A \emph{H\'enon-like map} $f$ on $D$
is a holomorphic map whose graph $\Gamma\subset D\times D$ satisfies the following properties:
\begin{enumerate}
    \item $\Gamma$ is a submanifold of $D\times D$ of pure dimension $k$ (not necessarily connected);
    \item $\pi_1|_\Gamma$ and $\pi_2|_\Gamma$ are injective;
    \item $\overline \Gamma$ does not intersect ${\partial_v D} \times \overline D$ and
    $\overline D \times  {\partial_h D}$.
 \end{enumerate}
\end{definition}

Let $\pi_M$ and $\pi_N$ be the canonical projections from $D$ to $M$ and $N$, respectively. A subset $B\subset D$ is said to be {\it horizontal} (resp. {\it vertical}) if $\pi_N(B)\Subset N$ (resp. $\pi_M(B)\Subset M$). Define $D_v:=\pi_1(\Gamma)$ and $D_h:=\pi_2(\Gamma)$ which are the domain and image of $f$. Then they are vertical and horizontal subsets of $D$, respectively. By definition, we can find bounded convex domains $M''\Subset M' \Subset M$ and $N''\Subset N'\Subset N$ such that 
\begin{equation}\label{eq:M''}
    f^{-1}(D)\subset M''\times N \quad \text{and} \quad f(D)\subset M\times N''.
\end{equation}
 Explicitly, if $f$ is a H\'enon-like map, $f$ and $f^{-1}$ are given by 
\[
    f(x)=\pi_2(\pi_1^{-1}(x)\cap\Gamma) \qquad\text{ and }\qquad f^{-1}(x)=\pi_1(\pi_2^{-1}(x)\cap\Gamma).
\]
Let $\tau:(x,y)\mapsto (y,x)$ be the involution map on $D\times D$. Then the graph of $f^{-1}$ is given by $\tau(\Gamma)$ and $f^{-1}$ is also a H\'enon-like map. We can iterate $f$ and consider its dynamics. Notice that both sequences $\{f^{-n}(D)\}$ and $\{f^n(D)\}$ are decreasing, so we have $f^{-n}(D)\subset M''\times N$ and $f^n(D)\subset M\times N''$ for all $n\geq 1$. 

We say {that} a current $T$ is {\it horizontal} (resp. {\it vertical}) on 
$D$ if its support is horizontal (resp. vertical). The cone of all positive closed horizontal currents of bi-degree $(k-p,k-p)$ on $D$ is denoted by $\mathscr C_h(D)$. For any $T\in\mathscr C_h(D)$, recall from \cite{DS06AIF} that we can define the {\it slice measure} $\lp T,\pi_1,a\rp$ for every $a\in M$ whose mass is independent of $a$ and is denoted by $\|T\|_h$. Explicitly, the {\it slice mass} is given by $\|T\|_h:=\lp (\pi_M)_*(T),\Omega_M\rp$ for any smooth probability measure $\Omega_M$ with compact support in $M$. The subset of $\mathscr C_h(D)$ {consisting of elements with} slice mass 1 is denoted by $\mathscr C_h^1(D)$. Similarly, we can define the sets $\mathscr C_v(D) $, $ \mathscr C_v^1(D)$, and the slice mass $\|\cdot\|_v$. {Take} $T\in\mathscr C_v(D)$ and $S\in\mathscr C_h(D)$. The intersection $T\wedge S$ is a well-defined positive measure of mass $\|T\|_v\|S\|_h$ with support in $\supp(T)\cap\supp(S)$, see \cite{DS06AIF}.   

\begin{example}\rm\label{example:ver-strict-posi}
    We construct some {smooth} vertical positive currents that will be useful later. We use the coordinates $(x_1,x_2)$ for points in $D=M\times N$. Denote by $\omega_M$ and $\omega_D$ the standard K\"ahler forms on $M$ and $D$, respectively. 

    {Let $M_1$ and $M_2$ be convex open domains of $M$ such that $M_1\Subset M_2\Subset M$.} Let $\rho:\C^p\to [0,1]$ be a smooth cut-off function which is supported in {$M_2$} and equals 1 on {$M_1$}. {For any $0\leq s\leq p,$ define $\Omega_{M,s}:=\rho\cdot\omega_M^s$, which is a smooth positive $(s,s)$-form on $M$.}

    Let $ A$ be a $p\times (k-p)$ matrix. We define the projection 
    {$\pi_{M, A}:\C^p\times\C^{k-p}\to \C^p$} by $\pi_{M, A}(x_1,x_2):=x_1+ A\cdot x_2$. {Then there exists $r>0$ depending only on $M_1$, $M_2$ such that when $\|A\|<r$, $\pi_{M,A}$ is well-defined on $M_1\times N$, with image in $M_2$.} Moreover, we can find a finite sum $\sum c_i\pi^*_{M, A_i}(\Omega_{M,s})$ for some $c_i>0$ and $ A_i$ small enough such that we have 
    \[
        \sum c_i\pi^*_{M, A_i}(\Omega_{M,s})\geq \omega_D^{s} \quad \text{on } M_1\times N. 
    \]
    {For details, see the proof of \cite[Lemma 3.5]{BDR23}. Note that when $s=p$, $\Omega_{M,p}$ is closed and so is the above finite sum. For $p<s'\leq k$, we can also obtain positive closed vertical smooth forms of bi-degree $(s',s')$ which are strictly positive on $M_1\times N$ by considering $\sum c_i\pi^*_{M, A_i}(\Omega_{M,p})\wedge\omega_D^{s'-p}$.}
\end{example}

\begin{lemma}\label{masscontolhoriz}
Let $T$ be a positive closed horizontal current of bi-dimension $(s,s)$ in $M\times N'$ {with $s\leq p$}. {Let $M_1$ and $M_2$ be convex open subsets of $M$ such that $M_1\Subset M_2\Subset M$.} Then there exist positive constants $C$ and $r$, {depending only on $M_1$, $M_2$, and $N'$,} such that for any {$p\times (k-p)$ matrix} $A$ with $\|A \|<r$, we have
$$
 \| (\pi_{M,A})_{*}(T)\|_{M_1} \leq C \| (\pi_{M})_{*}(T)\|_{M_2}.
$$ 
\end{lemma}

\begin{proof}
    We can always cover $M$ and $N$ by finitely many open balls and thus assume $M_1$, $M_2$, $M$, $N'$, and $N$ are all open balls centered at the origin. In particular, we can take {$M_1=\B_p(0,\vep)$ and $M_2=\B_p(0,4\vep)$} for some $\vep>0$. {We choose $r>0$ as in \cref{example:ver-strict-posi}. }
    
     {Take} $u=\|x_1\|^2$ and $u_{A}=\|x_1+A\cdot x_2\|^2$ which are smooth positive plurisubharmonic (psh) functions on $D$. It is immediate to see that $\pi_{M}^{*}(\omega_M)=\ddc u$ and $\pi_{M,A}^{*}(\omega_M)=\ddc u_A$. 
     Consider the continuous psh function given by
     $$
         v_A:=\max \{ u_A, (1+C_1)(u-4\vep^2) \}.
     $$ 
     We can choose {an} appropriate positive constant $C_1$ such that $v_A=u_A$ inside $\B_p(0,2\vep)\times N'$ and $v_A=(1+C_1)(u-4\vep^2)$ outside $\B_p(0,3\vep)\times N'$. In fact, by direct calculation, we can take $C_1=C_0(\vep r+r^2)\vep^{-2}$ for some constant $C_0>0$ only depending on $N'$. 
    Let $\chi$ be a cut-off function which equals $1$ inside $\B_p(0,3\vep)$ and is 0 outside $\B_p(0,4\vep)$. By {our choice of $v_A$ and $\chi$}, we have
    \[
        \|(\pi_{M,A})_{*}(T) \|_{M_1}= \int_{\B_p(0,\vep)\times N'} T\wedge (\ddc u_A)^s\leq \int \pi_M^{*}( \chi) T\wedge (\ddc v_A)^s.
    \]
    {Notice that $v_A=(1+C_1)(u-4\vep^2)$ on $\supp(\d\chi)$.} Then, by Stokes' formula, {we have} 
     \begin{align*}
          \int \pi_M^{*}( \chi) T\wedge (\ddc v_A)^s
          &=(1+C_1)^s \int \pi_M^*( \chi) T\wedge (\ddc u)^s\\
         &\leq (1+C_1)^s \int_{\B_p(0,4\vep)\times N'} T\wedge (\ddc u)^s = (1+C_1)^s\| (\pi_{M})_{*}(T) \|_{M_2}.
   \end{align*} 
   {The proof is completed by taking $C=(1+C_1)^s$ and $r$ as before.}
\end{proof}

\begin{definition}\label{defi:degrees}\rm
    For each $0\leq l\leq p$, {\it the dynamical degree }$d_l^+$ is defined by
    \[
        d_l^+:=\limsup_{n\to\infty}\left\{\sup_T\|(f^n)_*T\|_{M'\times N}\right\}^{1/n}
    \]
    where $T$ runs over all positive closed horizontal currents of bi-dimension $(l,l)$ of mass 1 on $D'{:=M'\times N'}$. 

    Similarly, for each $0\leq l\leq k-p$, {\it the dynamical degree }$d_l^-$ is defined by
    \[
        d_l^-:=\limsup_{n\to\infty}\left\{\sup_S\|(f^n)^*S\|_{M\times N'}\right\}^{1/n}
    \]
    where $S$ runs over all positive closed vertical currents of bi-dimension $(l,l)$ of mass 1 on $D'$. 
\end{definition}
 
The dynamical degrees defined above are independent of the choice of $M'$ and $N'$ \cite{DNS}. Moreover, $d_p^+=d_{k-p}^-:=d$ is an integer called {\it the main dynamical degree}. It is also characterized by the property that for any $T\in\mathscr{C}_h^1(D)$, $f_*(T)$ has slice mass $d$. The {\it Green current} $T^+\in\mathscr C_v^1(D)$ (resp. $T^-\in\mathscr{C}_h^1(D)$) is the unique positive closed vertical (resp. horizontal) current of slice mass 1 such that $f^*T^+=dT^+$ (resp. $f_*T^-=dT^-$). The following lemma allows us to compute the dynamical degrees on $M$. This will be useful for later purpose.

\begin{lemma}\label{lemma:mass-hori}
    For each $0\leq l\leq p$, we have 
     \[
        d_l^+=\limsup_{n\to\infty}\left\{\sup_T\|(\pi_M\circ f^n)_*T\|_{M'}\right\}^{1/n}
    \]
    where $T$ runs over all positive closed horizontal currents of bi-dimension $(l,l)$ of mass 1 on $D'$. The same is true for the $d^-_l$'s.
\end{lemma}

\begin{proof}
    {Since $\|(\pi_M\circ f^n)_*T\|_{M'}\leq \|(f^n)_*T\|_{M'\times N}$ for any $T$ considered in the definition of $d_l^+$, we have $d_l^+\geq\limsup_{n\to\infty}\left\{\sup_T\|(\pi_M\circ f^n)_*T\|_{M'}\right\}^{1/n}$. On the other hand, choose a convex open domain $M^*$ such that $M''\Subset M^*\Subset M'$. Let $\Omega_M$ be a positive $(l,l)$-form on $M$ such that $\Omega_M$ is strictly positive on $M''$ and $\supp(\Omega_M)\subset M^*$. By \cref{example:ver-strict-posi}, we can find finite number of $(c_i,A_i)$'s such that $\sum c_i\pi_{M,A_i}^*(\Omega_M)>\omega_D^l$ on $M''\times N$. Then we have 
    \begin{equation*}
        \begin{aligned}
            \|(f^n)_*T\|_{M''\times N}
            &=\lp (f^n)_*T,\omega_D^l|_{M''\times N}\rp\leq \sum c_i\lp (f^n)_*T,\pi_{M,A_i}^*(\Omega_M)\rp \\
            &\lesssim\sum c_i \|(\pi_{M,A_i}\circ f^n)_*T\|_{M^*}\lesssim \|(\pi_M\circ f^n)_*T\|_{M'}.
        \end{aligned}
    \end{equation*}
    The last inequality comes from \cref{masscontolhoriz} applied to $M_1=M^*$ and $M_2=M'$. The proof is completed since we have $d_l^+=\limsup_{n\to\infty}\left\{\sup_T\|(f^n)_*T\|_{M''\times N}\right\}^{1/n}$.
    } 
\end{proof}

\begin{example}\rm\label{example:prod}
    Let $f_1$ and $f_2$ be two H\'enon-like maps on $D_1=M_1\times N_1$ and $D_2=M_2\times N_2$, respectively. Suppose their main dynamical degrees are $d_1$ and $d_2$. Up to a permutation of coordinates, we can identify $D_1\times D_2$ with $(M_1\times M_2)\times (N_1\times N_2)$. Then it is easy to prove that the product map {$(f_1, f_2)$} is a H\'enon-like map on $D_1\times D_2$ whose main dynamical degree is $d_1d_2$. Suppose the Green currents associated to $f_i$ are $T_i^\pm$. Then the Green currents of {$(f_1, f_2)$} are $\mathbb T^+:=T_1^+\otimes T_2^+$ and $\mathbb T^-:=T_1^-\otimes T_2^-$. {Here the tensor product of two currents is defined in the usual sense: let $\pi_1$ and $\pi_2$ be the projection maps from $D_1\times D_2$ to its factors. For smooth test forms $\alpha_1$ on $D_1$ and $\alpha_2$ on $D_2$, we define 
    \[
        \lp T_1^+\otimes T_2^+, \pi_1^*(\alpha_1)\wedge\pi_2^*(\alpha_2)\rp:=\lp T_1^+,\alpha_1\rp\lp T_2^+,\alpha_2\rp,
    \]
    where only the components of \(\alpha_i\) of the complementary bi-degrees
    to \(T_i^+\) contribute to the pairings; all other components are understood
    to give zero. The action is then extended by linearity and continuity to all smooth test forms on $D_1\times D_2$. For more details, see \cite{D12Book,federer}.
    }
\end{example}

\subsection{Convergence to the Green currents}

 By \cite{DNS}, when $f$ satisfies \massump, $d^{-n}(f^n)^*(T)$ converges to $T^+$ uniformly on $T\in\mathscr{C}_v^1(D)$. We prove an equidistribution result for {not necessarily closed} smooth forms. 

\begin{proposition}\label{prop:non-closed-T^+}
    Suppose $f$ satisfies \massump. Let $\alpha$ be a positive but not necessarily closed smooth $(p,p)$-form with vertical support on $D$. Assume that $\lp T^-,\alpha\rp=1$. Then $d^{-n}(f^n)^*(\alpha)$ converges weakly to $T^+$.
\end{proposition}

\begin{proof} We will prove the assertion in four steps.\\

\noindent\textit{\textbf{Step 1}: reduction to the case $\alpha=\chi\alpha_0$.}
\medskip

    {Let $\Omega_M$ be a smooth positive $(p,p)$-form compactly supported in $M$ and strictly positive in a neighborhood of $\pi_M(\supp(\alpha))$. For $p\times (k-p)$ matrix $A$ with small entries, consider the form $\Omega_{M,A}:=\pi_{M,A}^*(\Omega_M)$ constructed in \cref{example:ver-strict-posi}.} Then, {by \cref{example:ver-strict-posi}}, we can choose a finite number of matrices $A_i$ such that for any point $x$ in a neighbourhood of $\supp(\alpha)$, the set $\{\Omega_{M,A_i}(x)\}$ is a base of the real cone of positive forms in $\bigwedge^pT_x^*\mathbb C^k$. Therefore, {locally,} $\alpha$ can be written as a linear combination of $\chi_i\Omega_{M,A_i}$ with some smooth positive functions $\chi_i$. We can multiply $\Omega_{M,A_i}$ {by a constant} so that $0\leq \chi_i\leq 1$. {Hence,} {by linearity}, it suffices to prove the result for $\alpha=\chi\alpha_0$ for some positive closed vertical $(p,p)$-form $\alpha_0$ and smooth function $\chi:D\to [0,1]$.
    
    Since $\alpha\leq \alpha_0$, the mass of $d^{-n}(f^n)^*(\alpha)$ {on any relatively compact subset of $D$} is bounded {independently of $n$. Hence, the set of positive currents $\{d^{-n}(f^n)^*(\alpha)\}_{n\in\N}$ is relatively compact.} Let $T_0$ be a limit value of $d^{-n}(f^n)^*(\alpha)$, and suppose $T_0=\lim_{i\to\infty}d^{-n_i}(f^{n_i})^*(\alpha)$. Then $\ddc T_0$ is a limit value of $d^{-n}(f^n)^*(\ddc\alpha)$. Take $f_0:=(f,f)$, which is still a H\'enon-like map on $D\times D$ by \cref{example:prod}. We have $T_0\otimes T_0=\lim_{i\to\infty} d^{-2n_i}(f_0^{n_i})^*(\alpha\otimes\alpha)$. \\

\noindent\textit{\textbf{Step 2:} both $T_0$ and $T_0\otimes T_0$ are $\ddc$-closed.}
\medskip

        By choosing a {suitable} positive closed vertical $(p+1,p+1)$-form which is strictly positive on the support of $\alpha$ {as in \cref{example:ver-strict-posi}}, we can write $\ddc\alpha$ as a difference of two positive closed vertical $(p+1,p+1)$-forms. Then \massump\ implies {that} the mass of $d^{-n}(f^n)^*(\ddc\alpha)$ converges to 0. Therefore $T_0$ is $\ddc$-closed. To prove {$T_0\otimes T_0$ is also $\ddc$-closed}, notice that
        \[
            \ddc (\alpha\otimes \alpha)=(\ddc \alpha)\otimes \alpha+ \alpha\otimes (\ddc \alpha)+\frac{\sqrt{-1}}{\pi}(\partial \alpha\otimes \bar{\partial}\alpha-\bar{\partial}\alpha\otimes \partial \alpha).
        \]
        After applying $(f_0^n)^*$ to both sides and taking limit, {by Step 1,} it suffices to prove that $d^{-n}(f_0^n)^*(\partial \alpha\otimes \bar{\partial}\alpha)$ goes to 0 when $n$ tends to infinity. Recall that $\alpha=\chi\alpha_0$. Therefore, {we have} $d^{-2n}(f_0^n)^*(\partial \alpha\otimes \bar{\partial}\alpha)=d^{-n}(f^n)^*(\partial\chi\wedge\alpha_0)\otimes d^{-n}(f^n)^*(\bar{\partial}\chi\wedge\alpha_0)$. An application of Cauchy-Schwarz inequality {together with \massump} shows that both $d^{-n}(f^n)^*(\partial\chi\wedge\alpha_0)$ and $d^{-n}(f^n)^*(\bar{\partial}\chi\wedge\alpha_0)$ converge weakly to zero as $n\to\infty$.\\

\noindent\textit{\textbf{Step 3:} $T_0$ is $\d$-closed.}
\medskip

    We apply the method in \cite{dTH-Dinh12AM}. Notice that 
    \[
        \ddc (T_0\otimes T_0)=(\ddc T_0)\otimes T_0+ T_0\otimes (\ddc T_0)+\frac{\sqrt{-1}}{\pi}(\partial T_0\otimes \bar{\partial}T_0-\bar{\partial}T_0\otimes \partial T_0).
    \]
    By Step 2, we deduce that 
    \[
        \partial T_0\otimes \bar{\partial}T_0=\bar{\partial}T_0\otimes \partial T_0.
    \]
    Notice that the bi-degree on the left hand side is $(p+1,p)$ over the first coordinate and $(p,p+1)$ over the second. But on the right hand side the bi-degree over the first coordinate is $(p,p+1)$. Hence we must have $\partial T_0\otimes \bar{\partial}T_0=\bar{\partial}T_0\otimes \partial T_0=0$. Let $\beta$ be any smooth $(k-p-1,k-p)$-form on $D$ with compact support. Then, since $T_0$ is real, {i.e., $T_0=\overline{T_0}$,} we have
    \[
        |\lp \partial T_0,\beta\rp|^2=|\lp \partial T_0\otimes \bar{\partial}T_0, \beta\otimes \bar{\beta}\rp|=0.
    \]
    This implies $T_0$ is $\partial$-closed. {In a similar way, we can get that $T_0$ is $\bar{\partial}$-closed. }\\ 

\noindent\textit{\textbf{Step 4:} conclusion.}
\medskip

    {Step 3 implies that} all limit values of $d^{-n}(f^n)^*(\alpha)$ {are $\d$-closed}. For any $l\in\mathbb N$, suppose $T_l$ is a limit value of $d^{-n_i+l}(f^{n_i-l})^*(\alpha)$. Then we have $T_0=d^{-l}(f^l)^*(T_l)$ for all $l\in\mathbb N$. Let $\beta$ be a positive closed horizontal smooth $(k-p,k-p)$-form of slice mass 1. Then the slice mass of $T_l$ can be computed by 
    \[
       \|T_l\|_v=\lp T_l,\beta\rp=\lim_{i\to\infty}\lp d^{-n_i+l}(f^{n_i-l})_*\beta,\alpha\rp=\lp T^-,\alpha\rp=1.
    \]
    Therefore {we have} $T_l\in\mathscr{C}_v^1(D)$ for all $l\in\N$. We can {then} apply uniform convergence on $\mathscr{C}_v^1(D)$ (see \cite[Theorem 4.6]{DNS})  and deduce that $T_0=\lim_{l\to\infty} d^{-l}(f^l)^*(T_l)=T^+$.
\end{proof}

\subsection{Super-potentials of positive closed vertical currents}
In the rest of the paper, we denote by $D':=M'\times N'$. Horizontal and vertical currents on $D'$ are defined similarly as on $D$.

Let $\PSH_h(D')$ be the set of horizontal currents $\Phi$ on $D'$ of bi-dimension $(p,p)$ such that $\ddc\Phi\geq 0$. We denote by $\DSH_h(D')$ the real vector space spanned by $\PSH_h(D')$. For any $\Phi\in\DSH_h(D')$, we define
\[
    \|\ddc\Phi\|_*:=\inf\{\|\Psi_1\|_{D'}+\|\Psi_2\|_{D'}\}
\]
where the infimum is taken over all $(k-p+1,k-p+1)$-currents $\Psi_1 $, $ \Psi_2$ which are positive closed horizontal in $D'$ and $\ddc\Phi=\Psi_1-\Psi_2$. 

The topology on $\DSH_h(D')$ is given by the following: a sequence $\{\Phi_n\}$ converges to $\Phi$ in $\DSH_h(D')$ if $\Phi_n\to\Phi$ weakly and $\|\ddc\Phi_n\|_*$ is bounded from above by some constant independent of $n$.  

\begin{definition}\rm
    Let $T$ be a positive closed $(p,p)$-current with vertical support in $M''\times N$. We say {that} $T$ has {\it continuous super-potentials}, if $T$ can be extended to a continuous linear functional on $\DSH_h(D')$ with respect to the topology introduced above.
\end{definition}

We can define continuous super-potentials for positive closed $(k-p,k-p)$-currents with horizontal support similarly. It is immediate to see that smooth positive closed $(p,p)$-forms with vertical support in $M''\times N$ always have continuous super-potentials. 

Let $T$ be a positive closed vertical $(p,p)$-current in $M''\times N$ and $S$ be a positive closed horizontal $(k-p,k-p)$-current in $M\times N''$. Assume that $T$ has continuous super-potentials. Then the intersection $T\wedge S$ {can be defined by $\lp T\wedge S,\varphi\rp:=\lp T,\varphi S\rp$ for any test function $\varphi$. The right hand side is well-defined since $\varphi S\in\DSH_h(D')$. It is shown in \cite{DS06AIF} that $T\wedge S$ is a} positive measure and coincides with the positive measure given in the discussion before \cref{example:ver-strict-posi}. 

{The following is proved in \cite{DNS}. Recently, under the same assumption, the stronger
H\"older continuity of the super-potentials is established in \cite{BDR24}. In this paper, we only need the continuity.}

\begin{theorem}\label{thm:Green-contin-sp}
    Suppose $f$ satisfies \massump, then the Green currents $T^\pm$ of $f$ have continuous super-potentials.
\end{theorem}

We define the equilibrium measure $\mu:=T^+\wedge T^-$. When $f$ satisfies \massump, $\mu$ is exponential mixing, hyperbolic, and of maximal entropy, see \cite{BDR24,DNS,DS06AIF}. As a {consequence of the above theorem and \cref{prop:non-closed-T^+}}, we {have the following proposition}.

\begin{proposition}\label{prop:non-closed-mu}
    Let $f$ and $\alpha$ be as in \cref{prop:non-closed-T^+}. Then $d^{-n}(f^n)^*(\alpha)\wedge T^-$ converges weakly to $\mu$. 
\end{proposition}

\begin{proof}
     Let $\varphi$ be a smooth test function with compact support on $D$. We may assume $0\leq \varphi\leq 1$. Then {we have} 
    \begin{equation}\label{eq:2.10-1}
        \lp d^{-n}(f^n)^*(\alpha)\wedge T^-,\varphi\rp=\lp T^-,  \varphi\cdot d^{-n}(f^n)^*(\alpha)\rp.
    \end{equation}
    By \cref{prop:non-closed-T^+}, $\varphi\cdot d^{-n}(f^n)^*(\alpha)$ converges weakly to $\varphi T^+$. {We} prove {that} $\|\ddc(\varphi\cdot d^{-n}(f^n)^*(\alpha))\|_*$ is uniformly bounded for all $n$. {Then, since $T^-$ has continuous super-potentials, the right hand side of \eqref{eq:2.10-1} converges to $\lp T^-,\varphi T^+\rp $, which equals $ \lp \mu,\varphi\rp$ by definition.} For simplicity, for any current $T$ and positive current $S$, we will write $|T|\leq S$ to denote that $-S\leq T\leq S$. We will prove that for each $n\geq 1$, there exist positive closed vertical $(p+1,p+1)$-{currents} $S_n$ such that $|\ddc(\varphi\cdot d^{-n}(f^n)^*(\alpha))|\leq S_n$ and $\|S_n\|_{D'}$ is bounded {independently} of $n$. Since $\ddc(\varphi\cdot d^{-n}(f^n)^*(\alpha))+S_n\leq 2S_n$ and is positive closed vertical, {this is enough to conclude the proof.} 
    
    {As in the} proof of \cref{prop:non-closed-T^+}, we can assume $\alpha=\chi\alpha_0$, where $\alpha_0$ is a positive closed vertical $(p,p)$-form and $\chi:D\to [0,1]$ is a smooth function on $D$. We have
    \begin{equation}\label{eq:non-close-converge}
         \begin{aligned}
             \ddc\big(\varphi\cdot d^{-n}(f^n)^*(\alpha)\big)
             &=\ddc\varphi\wedge d^{-n}(f^n)^*(\alpha)+\varphi\cdot d^{-n}(f^n)^*(\ddc\alpha)\\
             &-\frac{\sqrt{-1}}{\pi}\bar{\partial}\varphi\wedge d^{-n}(f^n)^*(\partial\chi\wedge\alpha_0)\\
             &+\frac{\sqrt{-1}}{\pi}{\partial}\varphi\wedge d^{-n}(f^n)^*(\bar{\partial}\chi\wedge\alpha_0).
         \end{aligned}
    \end{equation}
    
    Let $c$ be a positive constant such that $|\ddc\varphi|\leq c\omega_D$, $|\ddc\chi|\leq c\omega_D$, $\sqrt{-1}\partial \varphi\wedge\bar{\partial}\varphi\leq c\omega_D$, and $\sqrt{-1}\partial\chi\wedge\bar{\partial}\chi\leq c\omega_D$. Then {we have} 
    \[
       |\ddc\varphi\wedge d^{-n}(f^n)^*(\alpha)|\leq c\omega_D\wedge d^{-n}(f^n)^*(\alpha_0)
    \]
    and 
    \[
        |\varphi\cdot d^{-n}(f^n)^*(\ddc\alpha)|\leq c d^{-n}(f^n)^*(\omega_D\wedge \alpha_0).
    \]
    For the third term {in the right hand side of \eqref{eq:non-close-converge}}, by Cauchy-Schwarz inequality we have 
    \begin{equation*}
        \begin{aligned}
            |\sqrt{-1}\bar{\partial}&\varphi\wedge d^{-n}(f^n)^*(\partial\chi\wedge\alpha_0)|\\
            &\leq \frac{1}{2}d^{-n}\left(\sqrt{-1}\partial\varphi\wedge\bar{\partial}\varphi\wedge (f^n)^*(\alpha_0)+(f^n)^*(\sqrt{-1}\partial\chi\wedge\bar{\partial}\chi\wedge\alpha_0)\right)\\
            &\leq \frac{c}{2}d^{-n}\left(\omega_D\wedge (f^n)^*(\alpha_0)+(f^n)^*(\omega_D\wedge\alpha_0)\right).
        \end{aligned}
    \end{equation*}
    The last term { of \eqref{eq:non-close-converge}} can be bounded similarly. Finally, we choose $S_n=2cd^{-n}\big(\omega_D\wedge (f^n)^*(\alpha_0)+(f^n)^*(\omega_D\wedge\alpha_0)\big)$. {By \cref{defi:degrees}, $\|\omega_D\wedge d^{-n}(f^n)^*(\alpha_0)\|_{D'}$ is bounded independently of $n$. Since $\omega_D\wedge \alpha_0$ is of bi-degree $(p+1,p+1)$, it follows from \massump\ that $\|d^{-n}(f^n)^*(\omega_D\wedge \alpha_0)\|_{D'}$ converges to 0 exponentially fast. We deduce that $\|S_n\|_{D'}$ is bounded independently of $n$.} 
\end{proof}

Given a positive closed horizontal current, by convolution we can always construct a smooth positive closed form {which is} $\ddc$-cohomologous to it with good support. We end this section by a concise demonstration of this {fact}.

Let $\delta_0$ be the Dirac measure at $0\in\C^m$. Let $\nu_\vep$ be the smooth uniform probability measure on the sphere $S_{\vep}$ of radius $\vep$ centred at $0$. Notice that $\delta_0=(\ddc \log \|x\|)^m$ and $\nu_{\vep}=(\ddc \log \max(\|x\|,\vep ))^m$. Define 
\[
    \phi_{\vep}:=\log \|x\|(\ddc \log \|x\|)^{m-1}-\log \max(\|x\|,\vep )\big(\ddc \log \max(\|x\|,\vep )\big)^{m-1}.
\]
Then $\phi_\vep$ is a negative $L^1$-form supported in the closed ball of radius $\vep$ centred at 0. Moreover, we have $\delta_0-\nu_\vep=\ddc\phi_\vep$. Finally, by choosing a local coordinate chart, we can extend this {construction} to any Dirac measure on a complex manifold.

\begin{lemma}\label{lemma:convolution-T}
Let $T$ be a positive closed horizontal $(k-p,k-p)$-current in $M\times N'$. Then there exist a smooth positive closed $(k-p,k-p)$-form $\alpha_0$ and a negative $L^1$-form $U$, both defined on $M'\times N$ and supported on $D'$, such that
\[
    T=\alpha_0+\ddc U \text{ on } M'\times N.
\] 
Moreover, $\alpha_0$ is strictly positive in a small neighbourhood of $\supp(T)$.
\end{lemma}

\begin{proof}
    Let $ {\rm Aff}(\C^k)$ be the general affine group of $ \C^k$. Consider the map $\iota_0:\C^k\times {\rm Aff}(\C^k)\to \C^k$ given by $\iota_0(x,\tau)=\tau(x)$ for any $x\in\C^k $ and $ \tau\in{\rm Aff}(\C^k)$. Let $\nu$ be a probability measure on ${\rm Aff}(\C^k)$ supported in a small neighbourhood of $\id_{\C^k}$. Define $T_\nu:=(\iota_0)_*(T\otimes \nu)$. Then it is easy to verify that $T_\nu$ is well-defined and $T_\nu=\int_\tau \tau_*(T) \d\nu$.

    Consider a local coordinate chart of ${\rm Aff}(\C^k)$ centred at $\id_{\C^k}$. As per the previous discussion, we can find a smooth measure $\nu_\vep$ such that $\delta_{\id_{\C^k}}-\nu_\vep=\ddc\phi_\vep$. {Define} $\alpha_0:=(\iota_0)_*(T\otimes \nu_\vep)$. Since $T=(\iota_0)_*(T\otimes \delta_{\id_{\C^k}})$, we have $T-\alpha_0=\ddc U $ where $U=(\iota_0)_*(T\otimes \phi_\vep)$. Notice that $\alpha_0$ is a convolution of $T$ with respect to the measure $\nu_\vep$. When $\vep$ is small enough, all the remaining assertions are easy to check.
\end{proof}

\section{Tameness of Green currents}\label{sec:tame}

In {the} first part of this section, we briefly recall some results on woven currents and (almost) tame currents. The reader is referred to \cite[Section 3]{DS16} for {a more} detailed discussion. See also the survey \cite{Dujardin22} and the references there. Then we proceed to our main result: {under the assumption \massump,} the Green currents of a H\'enon-like map are tame. 

{Before we start, we introduce several notations {that} we will use later. We use $X$ to denote a compact K\"ahler manifold. Let $\G(p,k)$ denote the Grassmannian parametrizing all complex linear subspaces of $\C^k$ of dimension $p$. Suppose {$E$} is a holomorphic vector bundle over $X$. {We denote by $\Gr(E,p)$ the Grassmannian bundle over $X$ whose fibre over $x\in X$ is the Grassmannian of $p$-dimensional complex linear subspaces of the fibre $E_x$. Equivalently,} $\Gr(E,p)$ {is the} set of points $(x,[v])$ where $x\in X$ and $v$ is a $p$-dimensional subspace in $E_x$. For simplicity, when $E$ is the tangent bundle {$TX$ over} $X$, we {write} $\Gr(X,p)$ {for} $\Gr(TX, p)$ if {there is no ambiguity.} } 

\subsection{Woven, almost-tame, and tame currents}
{In this subsection, we work on} a K\"ahler manifold $(V,\omega)$ of dimension $l$, and fix $0\leq p\leq l-1$.

Let $\Sc_p(V) $ be the set of $p$-dimensional connected complex manifolds (not necessarily closed) $S$ in $V$ with locally finite mass. 
The last property means {that} for any compact set $K\subset V$ we have 
\[
    \int_{S\cap K}\omega^p<\infty.
\]
Each $S\in\Sc_p(V)$ is called a {\it lame}, and it defines a positive current $[S]$ of bi-dimension $(p,p)$ whose action on smooth $(p,p)$-forms is simply given by their integrations over $S$. We say {that} a positive measure $\nu$ on $\Sc_p(V)$ is {\it locally finite} if for any compact set $K\subset V$, 
\[
    \int_{S\in \Sc_p(V)}\lp [S], \omega^p|_K \rp \d\nu(S)<\infty.
\]
\begin{definition}\rm\label{defi:woven}
    Let $T$ be a $(l-p,l-p)$-current on $V$. Then $T$ is {\it woven} if there exists a locally finite positive measure $\nu$ on $\Sc_p(V)$ such that 
    \begin{equation}\label{equa:woven}
        T=\int_{S\in\Sc_p(V)}[S] \d\nu(S).
    \end{equation}
   We say {that} $\nu$ is a {\it woven measure} of $T$.
\end{definition}

{Note that for a woven current $T$, the woven measure $\nu$ may not be unique.} When a sequence of analytic subsets converges to a current, {it is natural to ask} under what condition the limit is woven. A criterion is given by de Th\'elin \cite{dTh} by considering the lift of currents to the Grassmannian bundle. When $S$ is a smooth complex manifold of dimension $p$, its {\it lift} to $\Gr(V,p)$ is given by
\begin{equation}\label{eq:mfd-lift}
    \widehat{S}=\{(x,[v]):x\in S, v \text{ represents }T_xS\}.
\end{equation}

Let $T$ be a woven current of bi-dimension $(p,p)$ on $V$ as in (\ref{equa:woven}). If $\int_{S\in\Sc_p(V)}\|[\widehat{S}]\|\d\nu(S)<\infty$, then $T$ admits a {\it lift }with respect to $\nu$. 
This is a positive current on $\Gr(V,p)$ of bi-dimension $(p,p)$ given by
\begin{equation}\label{eq:woven-lift}
    \widehat{T}=\int_{S\in \Sc_p(V)}[\widehat{S}]\d\nu(S).
\end{equation}

Note that the lift may depend on the choice of $\nu$ but its push-forward to $V$ is always $T$.

\begin{example}\rm\label{exam:gamma-hat}
    Recall that {we denote by $ {\rm Aff}(\C^k)$} the general affine group of $ \C^k$. Let $\rho$ be a smooth probability measure on ${\rm Aff}(\C^k)$ with $\supp(\rho)$ contained in a small neighbourhood of $\id_{\C^k}$. Choose a vertical $(k-p)$-dimensional affine subspace $L$ in $M''\times N$. For any $\tau\in \supp(\rho)$, {define} $L_\tau:=\tau(L)$ which is still vertical in $M''\times N$ when $\tau$ is close to $\id_{\C^k}$. We define two currents 
    \[
        \gamma: = \int_{\tau\in {\rm Aff}(\C^k)} [L_\tau] \d\rho(\tau) \quad \text{and} \quad  \widehat{\gamma}:= \int_{\tau \in {\rm Aff}(\C^k)} [\widehat{L}_{\tau}] \d\rho(\tau).
    \]
    When $\supp (\rho)$ is small enough, we have $\supp(\gamma) \subset M''\times N$ and {$\supp (\widehat{\gamma})$ is contained in a small neighborhood of $\widehat{L}$ in $M''\times N\times \G(k-p,k)$.} Both $\gamma$ and $\widehat{\gamma}$ are smooth positive closed forms, and $\widehat{\gamma}$ is a lift of $\gamma$. 
\end{example}

\begin{theorem}[\cite{dTh}] \label{thm:de-thelin}
    Let $\Sigma_n$ be a sequence of analytic subsets of pure dimension $p$ in $V$ and $\widehat{\Sigma}_n$ be the lift of $\Sigma_n$ to $\Gr(V,p)$. Let $v_n$ and $\hat{v}_n$ be the volume of $\Sigma_n$ and $\widehat{\Sigma}_n$, respectively. Suppose all $v_n$ and $\hat{v}_n$ are finite and $v_n^{-1}[\Sigma_n]$ converges to a current $T$. If $v_n^{-1}\hat{v}_n$ is bounded by some constant $c>0$ independent of $n$, then $T$ is woven.
\end{theorem}

Let $T$ be a woven positive closed $(l-p,l-p)$-current which admits a lift in $\Gr(V,p)$. Its lift is always positive but may not be closed. The following property of woven currents is introduced and studied in \cite{DS16}.
\begin{definition}\rm
    Let $T$ be a woven positive closed $(l-p,l-p)$-current on $V$. We say {that} $T$ is {\it almost tame} if $T$ admits a lift $\widehat{T}$ and there is a positive closed current $\widehat{T}'$ on $\Gr(V,p)$ such that $\widehat{T}'\geq \widehat{T}$ and the push-forward of $\widehat{T}'$ onto $V$ {is equal to} $T$. {In this case, we say that $\widehat{T}'$ \textit{dominates} a lift of $T$.} We say {that} $T$ is {\it tame} when it admits a lift $\widehat{T}$ which is closed itself. The measure $\nu$ associated with this lift is said to be {\it almost tame} or {\it tame}, respectively. 
\end{definition}

\begin{proposition}[\cite{DS16}, Proposition 3.14]\label{prop:altame-dethelin}
    Let $\Sigma_n,v_n,\widehat{\Sigma}_n,\hat{v}_n$, and $T$ be as in \cref{thm:de-thelin}. Suppose the volume of the lift of $\widehat{\Sigma}_n$ to $\Gr(\Gr(V,p),p)$ is bounded by $cv_n$ for some constant $c>0$ independent of $n$, then $T$ is almost tame. {More precisely, any limit of $\widehat{v}_n^{-1}[\widehat{\Sigma}_{n}]$ dominates a lift of $T$.}
\end{proposition}

Finally, we {will now explain} how the woven structure is inherited under a submersion. Let $\pi:V\to W$ be a {proper} holomorphic submersion onto a complex manifold $W$ of dimension larger {than} or equal to $p$. We define $\Sc_p^1(V)$ to be the set of all lames $S\in \Sc_p(V)$ such that the restriction of $\pi$ {to} $S$ is generically of rank $p$, and {set} $\Sc_p^2(V):=\Sc_p(V)\setminus\Sc_p^1(V)$.
Then for any woven positive $(l-p,l-p)$-current $T$ on $V$, we can decompose it as 
\[
    T=T_1+T_2,
\]
where $T_i=\int_{S\in\Sc_p^i(V)}[S]\d\nu(S)$ for $i=1,2$. Let $\pi':\Gr(V,p)\to V$ be the canonical projection map. Define 
        \begin{equation}\label{eq:non-transverse}
            Z:=\{(x,[v]): x\in V, v \text{ is not transverse to the fiber }(\pi')^{-1}(x) \},
        \end{equation}
    which is an analytic set of $\Gr(V,p)$. When $T$ is almost tame, let $\widehat{T}'$ be a positive closed current on $\Gr(V,p)$ which dominates a lift of $T$. Define
        \begin{equation}\label{eq:skoda}
            \widehat T_2:=\textbf{1}_Z\widehat T' \quad\text{and}\quad \widehat T_1:=\widehat T'-\widehat T_2.
        \end{equation}

By Skoda's theorem, they are both positive closed. The following result comes from \cite[Lemma 3.11]{DS16} and its proof. 

\begin{lemma}\label{lemma:woven after projec}
    {Let $\pi$, $T$, $T_1$, and $T_2$ be as above. Then $\pi_*(T_2)=0$ and $\pi_*(T_1)$ is still woven. When $T$ is almost tame and $\widehat{T}'$ dominates a lift of $T$, let $\widehat{T}_1$ and $\widehat{T}_2$ be as in \eqref{eq:skoda}. Then we have $T_i=(\pi')_*(\widehat T_i)$ for $i=1,2$. In particular, both $T_1$ and $T_2$ are closed. }
\end{lemma}

\subsection{Tameness of Green currents}
We consider a H\'enon-like map $f$ on $D=M\times N$. Notice that the Grassmannian bundle of $D$ is trivial: $\Gr(D,p)=D\times \G(p,k)$. Let $T^\pm$ be the Green currents of $f$. 
The following is shown by Bianchi-Dinh-Rakhimov \cite{BDR24}:

\begin{theorem}\label{thm:laminar}
    Let $f$ be a H\'enon-like map on $D$ satisfying \massump. Then the Green currents $T^\pm$ of $f$ are woven. 
\end{theorem}

We will apply the idea of the {proof of the} above theorem to show that $T^\pm$ is {actually} tame. In fact, we will prove a stronger result, see \cref{thm:tame}. For simplicity, we only treat the case of $T^-$. The proof for $T^+$ is completely symmetric. In the remaining part of this section, we use $X$ to denote a compact K\"ahler manifold whose K\"ahler form is denoted by $\omega_X$. The K\"ahler form of $D\times X$ is denoted by $\omega:=\pi_D^*(\omega_D)+\pi_X^*(\omega_X)$. Here $\pi_D$ and $\pi_X$ are the canonical projections from $D\times X$ to $D$ and $X$.

 Given a H\'enon-like map $f$ on $D$, it admits a {\it canonical lift} to $D\times \G(p,k)$ given by $\widehat{f}(z,v):=(f(z),\d f_z(v))$ for $z\in D$ and $v\in \G(p,k)$. We can extend this notion to a more general setting.

\begin{definition}\rm\label{defi:lift}
    Let $X$ be a compact K\"ahler manifold, and suppose $F:D_v\times X\to D_h\times X$ is bi-holomorphic. We say {that} $F$ is a {\it lift} of $f$ to $D\times X$ if $\pi_D\circ F=f\circ \pi_D$ and for every $z\in D_v$, $\pi_X\circ F(z,\cdot)$ induces a trivial action on the Hodge cohomology classes of $X$.
\end{definition}

The following lemma justifies the terminology: the canonical lift is indeed a lift. 

\begin{lemma}\label{lem:basicinva}
     $\widehat{f}$ is a lift of $f$ to $D\times \G(p,k)$.
\end{lemma}

\begin{proof}
    It suffices to prove that for every $z\in D_v$, $\d f_z$ induces a trivial action on the Hodge cohomology classes of $\G(p,k)$. Notice that $\d f_z$ is given by a matrix in $\Gl(k,\C)$. As $\Gl(k,\mathbb{C})$ is path connected, we can use a smooth curve to connect the identity matrix $\id_{\C^k}$ with $\d f_z$. This gives a homotopy between $\id_{\C^k}$ and $\d f_z$. Thus, $\d f_z$ acts trivially on ${H}^{*}(\mathbb{G}(p,k),\C)$.   
\end{proof}

Notice that $\Gr(D\times X, p)$ is the set of points $(z,w,H_w)$ where $z\in D, w\in X$ and $H_w$ represents a $p$-dimensional subspace in $\C^k\oplus T_wX$. Therefore, $\Gr(D\times X, p)$ can be identified with $D\times \Gr(\mathbb{C}^{k}\oplus TX,p)$. Here $\C^k\oplus TX$ is the direct sum of the trivial bundle $\C^k$ over $X$ with the tangent bundle of $X$. Let $F:D_v\times X \to D_h\times X$ be a lift of $f$. It has a canonical lift $\widehat{F}$ to $\Gr(D\times X, p)$ given by $\widehat{F}(z,w,H_w) :=\big(F(z,w),\d F_{(z,w)}(H_w)\big)$.

\begin{lemma} \label{lem: invacoho}
    Let $F$ and $\widehat{F}$ be as above. Then $\widehat{F}$ is also a lift of $f$ to $D\times \Gr(\mathbb{C}^{k}\oplus TX,p)$.
\end{lemma}

\begin{proof}
    Denote by $E=\C^k\oplus TX$ which is a holomorphic vector bundle over $X$. For every $z\in D_v$, we need to show that $\pi_\Gr\circ \widehat{F}(z,\cdot)$ preserves the Hodge cohomology classes of $\Gr(E,p)$. Here $\pi_\Gr$ is the projection to $\Gr(E,p)$. {Take} $\phi=\pi_X\circ F(z,\cdot)$ which is an automorphism of $X$. Define $\Phi:E\to E$ {the} isomorphism of $E$ given by $\Phi(w,v)=(
    \phi(w),\d F_{(z,w)}(v))$ for any $w\in X$ and $v\in \C^k\oplus T_wX$. Then $\Phi$ induces an automorphism of $\Gr(E,p)$ denoted by $\Phi'$. It is easy to check that $\Phi'=\pi_\Gr\circ \widehat{F}(z,\cdot)$. Let $\pi':\Gr(E,p)\to X$ be the canonical projection. Then $\pi'\circ \Phi'=\phi\circ \pi'$.

    Let $h$ be the Chern class of the canonical line bundle $L$ of $\Gr(E,p)$. Then {we have} $(\Phi')^*(h)=h$. By Leray-Hirsch theorem {(see, for example, \cite[Lemma 7.32]{Voisin-1})}, any element in $H^*(\Gr(E,p),\mathbb{C})$ can be expressed as a linear combination of classes of the form $h^l\smile (\pi')^*(c)$ where $c$ is a class in $ H^*(X,\mathbb{C})$. We can apply Leray-Hirsch theorem in this case as the Chern class of $L$ restricted to any fibre of $\Gr(E,p)$ generates multiplicatively the cohomology of the fibre. 
    
    We deduce that $(\Phi')^*(h^l\smile (\pi')^*(c))=h^l\smile (\pi')^*(\phi^*c)=h^l\smile (\pi')^*(c)$ since $\phi$ acts trivially on $H^*(X,\C)$. Thus $\Phi'$ acts trivially on $H^*(\Gr(E,p),\C)$.
\end{proof}

Now we discuss some properties of lifts of $f$ which are essential in our proof of \cref{thm:tame}. They are proven in \cite{BDR24} for canonical lifts, but the proofs extend to all lifts without difficulty.

First, we introduce the notions of the {\it $D$-dimension} and {\it shadow} of a positive current on $D\times X$. They are closely related to the h-dimension discussed in \cref{sec:sp}.

\begin{definition}[\cite{BDR24}, Definition 5.6] \rm\label{defshaw}
    Let $T$ be a nonzero positive (not necessarily closed) current of bi-dimension $(p,p)$ on $D\times X$. The {\it $D$-dimension} of $T$ is the maximal non-negative integer $l$ such that $T\wedge (\pi_D)^*(\omega_D^l)\neq 0$. The {\it shadow} of $T$ on $D$ is the nonzero positive current of bi-dimension $(l,l)$ on $D$ given by the formula
    \[
        \shad(T):=(\pi_D)_*\big(T\wedge (\pi_X)^*\omega_X^{p-l}\big).
    \]
\end{definition}

\begin{lemma}\label{lemma:lift-D-dimension}
    Let $F: D_{v}\times X \to D_{h}\times X$ be a lift of $f$. Let $T$ be a nonzero positive closed current of bi-dimension $(p,p)$ on $D\times X$ with $ \supp (T)\subset M\times N'\times X$. Then, {$F_*(T)$ is a positive closed current of bi-dimension $(p,p)$ whose support is contained in }$M \times N' \times X$. If $F_*(T)$ is non-zero, then the {\it $D$-dimension} of $F_*(T)$ is at most the {\it $D$-dimension} of $T$. Furthermore, if their {\it $D$-dimensions} are equal, we have that $\shad({F_{*}(T}))=f_{*}(\shad({T})).$   
\end{lemma}

\begin{proof}
    Since $F$ is bi-holomorphic and is a lift of $f$, it follows that $F_*(T)$ is a positive closed current of bi-dimension $(p,p)$ on $D\times X$ with its support contained in $M \times N'\times X$.
    
    Let $l$ be the $D$-dimension of $T$. Suppose $F_*(T)$ is non-zero and its $D$-dimension is at least $l+1$. Then, we can find a cut-off function $\chi$ on $D$ and a positive smooth $(p-l-1,p-l-1)$-form $\varphi$ on $D\times X$, both with compact support, such that $\lp F_*(T)\wedge \pi_D^*(\omega^{l+1}_D), \pi_D^*(\chi)\cdot\varphi\rp>0$. Note that we have 
    \[
        \lp F_*(T)\wedge \pi_D^*(\chi \omega^{l+1}_D) , \varphi\rp= \lp T, \pi^*_D(f^*(\chi \omega_D^{l+1}))\wedge F^*(\varphi) \rp =\lp T\wedge \pi^*_D(f^*(\chi \omega_D^{l+1})), F^*(\varphi) \rp.
    \]
    By \cite[Lemma 5.5]{BDR24}, we have $T\wedge \pi_D^{*}( \theta)=0$ for any smooth $(2l+1)$-form $\theta$ on $D$. Hence, the right hand side is 0, which is a contradiction. We conclude that $D$-dimension of $F_*(T)$ is at most $l$ when $F_*(T)$ is non-zero. 
    
    Finally, when the two $D$-dimensions are equal, the condition that $\pi_X\circ F(z,\cdot)$ is cohomologically trivial allows us to extend the proof of \cite[Lemma 5.16(iii)]{BDR24} to get the identity $\shad(F_*(T))=f_*(\shad(T))$.
\end{proof}

{The following result is proven in \cite[Proposition 5.9]{BDR24} for the canonical lift of $f$. We extend it to arbitrary lifts of $f$, with a constant independent of $S$.}

\begin{proposition}\label{prop:controlmass}
    Let $f$ be a H\'enon-like map on $D$ satisfying \massump. Let $Y$ be a compact K\"ahler manifold. Let $G:D_v\times Y\to D_h\times Y$ be a lift of $f$. Then there exists a constant $C>0$, depending {only} on $M'$ and $M''$, such that for any positive closed current $S$ on $D\times Y$ with $\supp(S)\subset M\times N'\times Y$ and any $n\geq 1$, we have
    \begin{equation}\label{inequality mass control}
        \|d^{-n} (G^n)_*S\|_{M''\times N\times Y}\leq C\|S\|_{M'\times N\times Y}.
    \end{equation}
\end{proposition}

\begin{proof}
    Let $S$ be a positive closed current of bi-dimension $(r,r)$ in $M\times N'\times Y$ with $\|S\|_{M'\times N\times Y}=1$. By \cref{lemma:lift-D-dimension}, $(G^n)_*S$ is of bi-dimension $(r,r)$ for any $n\geq 1$. Following the proof of \cite[Proposition 5.9]{BDR24} and {using the trivial action of $G(z, \cdot)$ on the Hodge cohomology of $Y$}, we have
    \[
        \|(G^n)_*S\|_{M''\times N\times Y}\leq Cd^n
    \]    
    for some constant $C>0$ and any $n\geq 1$.

    It remains to prove that $C$ can be chosen to be independent of $S$. Otherwise, we can find a sequence $\{(n_{i},S_{i})\}$ such that $\|S_{i}\|_{M'\times N\times Y}=1$ and $ \| (G^{n_{i}})_{*}(S_i)\|_{M''\times N\times Y}> 4^{i}d^{n_{i}}$. Define $S:=\sum_{i} 2^{-i}S_{i}$. Then {we have} $d^{-n_{i}}\| (G^{n_{i}})_{*}(S)\|_{M''\times N\times Y}>2^{i}$ for all $i$. This is a contradiction. 
\end{proof}

Let $\Sigma$ be a horizontal submanifold of dimension $p$ in $M\times N'$. We say {that} a submanifold $\Sigma'$ of dimension $p$ in $D\times X$ is a {\it lift} of $\Sigma$ if $\pi_D$ is a biholomorphic map onto $\Sigma$ when restricted on $\Sigma'$. When $X=\G(p,k)$, the lift defined in (\ref{eq:mfd-lift}) is referred to as the {\it canonical lift} of $\Sigma$ and is denoted by $\widehat{\Sigma}$. Again, the canonical lift of a submanifold is also a lift. If $F$ is a lift of $f$, then it is easy to see that ${\Sigma}'_n:=F^n({\Sigma}')$ is a lift of $\Sigma_n:=f^n(\Sigma)$. \cref{prop:controlmass} implies the mass of $d^{-n}[\Sigma'_n]$ is uniformly bounded from above for all $n$.

\begin{lemma}\label{lem almtame}
Let $\Sigma'_n$ be as above. Then any limit current of  $d^{-n}[\Sigma'_n]$ is almost tame.
\end{lemma}

\begin{proof}
     In order to apply \cref{prop:altame-dethelin}, we need to lift $\Sigma_n'$ twice to $\Gr(\Gr(D\times X,p),p)$. Recall that we have $\Gr(D\times X, p)=D\times X_1$ where $X_1=\Gr(\C^k\oplus TX,p)$. Then {we deduce that} $\Gr(\Gr(D\times X,p),p)=\Gr(D\times X_1,p):=D\times X_2$ where $X_2=\Gr(\C^k\oplus TX_1,p)$. Let $F_1$ be the canonical lift of $F$ to $D\times X_1$ and $F_2$ be the canonical lift of $F_1$ to $D\times X_2$. We apply \cref{lem: invacoho} twice: by applying to $F$, we deduce that $F_1$ is a lift of $f$. Then we apply it to $F_1$ and deduce that $F_2$ is also a lift of $f$.
     
     The canonical lift of $\Sigma'$ to $\Gr(D\times X,p)$ is denoted by $\Sigma''$ and the canonical lift of $\Sigma''$ to $\Gr(\Gr(D\times X,p),p)$ is denoted by $\Sigma'''$. Similarly, we define $\Sigma_n''$ and $\Sigma_n'''$ as lifts of $\Sigma_n$. It follows that $\Sigma_n''=F_1^n(\Sigma'')$ and $\Sigma_n'''=F_2^n(\Sigma''')$.
     
     {We} apply \cref{prop:controlmass} with {$X_2$ instead of $Y$ and $F_2$ instead of $G$}. Then we can find some constant $C>0$ independent of $n$ such that $\|d^{-n}\Sigma_n'''\|_{M''\times N\times X_2}\leq C$. In order to extend this to $D\times X_2$, notice that $\Sigma_n'''=F_2(\Sigma_{n-1}''')$ and (\ref{eq:M''}) implies that $F_2^{-1}(D\times X_2)=M''\times N\times X_2$. Therefore 
     \[
         \|\Sigma_n'''\|=\|F_2(\Sigma_{n-1}'''|_{M''\times N\times X_2})\|\lesssim\|\Sigma_{n-1}'''\|_{M''\times N\times X_2}\lesssim d^n.
     \]
    The result follows directly from \cref{prop:altame-dethelin}.
\end{proof}

Let $\Sigma$ and $\widehat{\Sigma}$ be as before. Define $\widehat{\Sigma}_n:=\widehat{f}^n(\widehat{\Sigma})$ which is the canonical lift of $\Sigma_n:=f^n(\Sigma)$. 

\begin{theorem}\label{thm:tame}
    Let $f$ be a H\'enon-like map on $D$ satisfying \massump. Then any limit current of $d^{-n}[\widehat{\Sigma}_n]$ is a lift of $T^-$. In particular, $T^-$ is tame. 
\end{theorem}

\begin{proof}    
    Suppose $d^{-n_i}[\widehat{\Sigma}_{n_i}]$ converges to $\widehat{T}$ for some subsequence $\{n_i\}$. It follows from \cref{lem:basicinva} and \cref{lem almtame} that $\widehat{T}$ is almost tame. By \cref{lemma:woven after projec} {with $V$ replaced by $\Gr(D,p)$}, we can decompose $\widehat{T}=\widehat{T}_1+\widehat{T}_2$ as a sum of two woven positive closed currents {of bi-dimension $(p,p)$} {such} that $(\pi_D)_*(\widehat{T}_2)=0$ and $(\pi_D)_*(\widehat{T}_1)$, which is exactly $T^-$, has a woven structure inherited from $\widehat{T}_1$. Therefore, $\widehat{T}_1$ is a lift of $T^-$. It remains to prove $\widehat{T}_2=0$.
    
   Let $\mathcal F$ be the set of all such $\widehat{T}_2$'s which comes from the decomposition of some limit current of $d^{-n}[\widehat{\Sigma}_n]$. 

   \medskip
   \noindent \textit{\textbf{Claim:}} there exists $\widehat{T}_2^{(1)}\in\mathcal F$ such that $ \widehat{T}_2=d^{-1}\widehat{f}_*(\widehat{T}^{(1)}_2)$.
   \medskip
   
  Assuming the claim, by induction, we can find a sequence of positive closed currents $\{\widehat{T}_2^{(m)}\}\subset\mathcal F$ such that {$\widehat{T}^{(m)}_{2}=d^{-1}\widehat{f} _*(\widehat{T}^{(m+1)}_2)$ for every $m\in\N$}. {Here, we take $\widehat{T}_2^{(0)}=\widehat{T}_2$. Denote by $l_m$ the $D$-dimension of $\widehat{T}_2^{(m)}$.} By \cref{lemma:lift-D-dimension}, {we have $l_m\leq l_{m+1}$ and $l_m<p$ for every $m\in\N$.} Therefore, we can find two integers $m_0 \geq 0 $ and $0\leq l < p$ such that {$l_m=l$} for any $m\geq m_0$. Hence, for every $m\geq m_0$, $\shad(\widehat{T}^{(m)}_2)$ is a  horizontal positive closed current of bi-dimension $(l,l)$. {Moreover}, \cref{lemma:lift-D-dimension} implies that for any $m\geq m_0 $, we always have 
    \begin{equation}\label{eq:invari-shad}
        \shad({\widehat{T}^{(m_0)}_{2}})=d^{-m+m_0}(f^{m-m_0})_*\big(\shad(\widehat{T}^{(m)}_2)\big).
    \end{equation}

    Since the mass of $\widehat{T}_2^{(m)}$ is uniformly bounded from above by some constant independent of $m$, by  \cref{defi:degrees}, we have 
  \[
      \limsup_{m\to\infty} \|(f^m)_*(\shad(\widehat{T}_2^{(m)}))\|_{M'\times N}^{1/m}\leq d_l^+<d.
  \]
  It then easily follows from \eqref{eq:invari-shad} that $\shad(\widehat{T}^{(m_0)}_2)=0$ and as a result $\widehat{T}^{(m_0)}_2=0$. Since $\widehat{T}_2=d^{-m_0}\widehat{f}_*(\widehat{T}^{(m_0)}_2)$, we conclude that $\widehat{T}_2=0$.
  \end{proof}

\begin{proof}[Proof of the claim]
    Consider the canonical lift $\wwhat{\Sigma}_n$ of $\widehat{\Sigma}_n$ to $\Gr(\Gr(D,p),p)$. By taking a subsequence, we may assume $d^{-n_i}[\wwhat{\Sigma}_{n_i}]$ converges to some woven positive closed current $\hhatT$. By \cref{prop:altame-dethelin}, $\hhatT$ dominates a lift of $\widehat{T}$. On the other hand, the same results hold also for the sequence $\{n_i-1\}$. Up to taking a subsequence, we may assume $d^{-n_i}[\widehat{{\Sigma}}_{n_i}]$ and $d^{-n_i+1}[\wwhat{\Sigma}_{n_i-1}]$ converges to $\widehat{T}^{(1)}$ and  
    $\hhatT^{(1)}$, respectively. Similarly, $\widehat{T}^{(1)}$ has the decomposition $\widehat{T}^{(1)}=\widehat{T}_1^{(1)}+\widehat{T}_2^{(1)}$ and $\hhatT^{(1)}$ dominates a lift of $\widehat{T}^{(1)}$. Let $\hhatf$ be the canonical lift of $\widehat{f}$ to $\Gr(\Gr(D,p),p)$. We have $\hhatT=d^{-1}\hhatf_*(\hhatT^{(1)})$.

    Now we take $V=\Gr(D,p)$ in \cref{lemma:woven after projec}. Let $Z$ be defined as in \eqref{eq:non-transverse}. We have the decompositions $\hhatT=\hhatT_1+\hhatT_2$ and $\hhatT^{(1)}=\hhatT^{(1)}_1+\hhatT^{(1)}_2$ as in \eqref{eq:skoda}, i.e., $\hhatT_2$ and $\hhatT_2^{(1)}$ are the restrictions of $\hhatT$ and $\hhatT^{(1)}$ to $Z$, respectively. Then we have 
    \[
        \hhatT=d^{-1}\hhatf_*(\hhatT^{(1)})=d^{-1}\hhatf_*(\hhatT^{(1)}_1)+d^{-1}\hhatf_*(\hhatT^{(1)}_2).
    \]
    Since $Z$ is totally invariant under $\hhatf$, we deduce that $\hhatT_2=d^{-1}\hhatf_*(\hhatT^{(1)}_2)$.
    Recall that $\pi':\Gr(V,p)\to V$ is the projection map, and it satisfies the equality $\pi'\circ \hhatf=\widehat{f}\circ\pi'$. The claim follows from \cref{lemma:woven after projec} and applying $\pi'$ to both sides of $\hhatT_2=d^{-1}\hhatf_*(\hhatT^{(1)}_2)$.
\end{proof}

\section{Super-potentials and density} \label{sec:sp}
In this section, we recall some elements of density theory developed {by} Dinh-Sibony \cite{DS18JAG} as a general intersection theory for positive closed currents. In the second part, we prove that the intersection induced from super-potentials is coherent with density theory. {The reader is referred to \cite{BDR24, DNS} for a detailed discussion about super-potentials in our setting; see also \cite{DNV,DS09AM, DS10JAG} in the compact setting.} 

\subsection{Density theory for positive closed currents}
Let {$( W,\omega)$} be a K\"ahler manifold of dimension $k$. Suppose {$Z\subset W$} is an $l$-dimensional complex submanifold. 

Let $\pi: N_{Z|W}\to Z$ denote the normal vector bundle to $Z$ in $W$. For $\lambda \in \C^{*}$, let $A_{\lambda}: N_{Z|W} \to N_{Z|W} $ be the multiplication by $\lambda$ on the fibres of $\pi$. So, $Z$ is invariant under the action of $A_{\lambda}$. Consider a diffeomorphism $\tau$ from a neighbourhood of $Z$ in $W$ to a neighbourhood of $Z$ in $N_{Z|W}$ such that the restriction of $\tau$ to $Z$ is the identity map. Consequently, $\d \tau$ induces a real isomorphism of $N_{Z|W}$. We choose a  $\tau$ such that this real isomorphism is the identity map. We call such $\tau$ an {\it admissible map}. In general, it is not holomorphic.

Fix an admissible $\tau$ as above. Let $T$ be a positive closed current of bi-degree $(p,p)$ on $W$ with no mass on $Z$ such that $\supp(T) \cap Z$ is compact. Consider the following family of closed currents of degree $2p$ in $N_{Z|W}$ indexed by $\lambda \in \C^{*}$:
$$
  T_{\lambda} = (A_{\lambda}\circ \tau)_{*}(T).
$$ 
Since $\tau$ may not be holomorphic, the current $T_{\lambda}$ may not be of bi-degree $(p,p)$ and we cannot talk about its positivity. However, for any sequence $\{ \lambda_n\}^{\infty}_{n=1}$ in $\C^{*}$ converging to infinity, there is a subsequence $\{\lambda_{n_j} \}^{\infty}_{j=1}$ such that $ T_{\lambda_{n_j}}$ converges to some positive closed current $\mathcal R$ of bi-degree $(p,p)$ in $N_{Z|W}$ as $j\to \infty$ (see \cite[Theorem 4.6]{DS18JAG}). We say {that} $\mathcal R$ is a {\it tangent current} of $T$ along $Z$. It may depend on the sequence $\lambda_{n_j}$ but is independent of the choice of $\tau$. {More precisely, if $\tau'$ is another admissible map and if $T_{\lambda_n}=(A_{\lambda_n}\circ \tau)_{*}(T)$ converges to $\mathcal{R}$ as $\lambda_n\to\infty$, then $(A_{\lambda_n}\circ \tau')_*(T)$ converges to the same limit current $\mathcal{R}$.}

Although tangent currents may not be unique, if we consider {the extension by zero} to $\overline{N}_{Z|W}$, they belong to the same cohomology class in $H^{2p}_{c}(\overline{N}_{Z|W}, \mathbb{R})$. Here $\overline{N}_{Z|W}$ denotes the projective compactification of $N_{Z|W}$ along the fibre direction, i.e., every fibre is bi-holomorphic to $\mathbb{CP}^{k-l}$, and $H_{c}^{*}(\cdot,\R)$ is the de Rham cohomology group defined by forms with compact support. We denote this cohomology class by $\kappa^{Z}(T)$ and call it {\it the total tangent class} of $T$ along $Z$. Let $-h_{\overline{N}_{Z|W}}$ be the tautological class on $\overline{N}_{Z|W}$ which is the Chern class of the tautological line bundle $O_{\overline{N}_{Z|W}}(-1)$ over $\overline{N}_{Z|W}$. Leray's theorem implies that the cohomology ring $\bigoplus H^{*}_{c}(\overline{N}_{Z|W},\R)$ is a free $\bigoplus H^{*}_{c}(Z,\R)$-module generated by the classes $1, h_{\overline{N}_{Z|W}},\cdots, h^{k-l}_{\overline{N}_{Z|W}}$. So, we can decompose $\kappa^{Z}(T)$ in a unique way as 
$$
    \kappa^{Z}(T)= \sum\limits_{j=\max(0,l-p)}^{\min (l,k-p)} \pi^{*}(\kappa^{Z}_{j}(T))\smile h^{j-l+p}_{\overline{N}_{Z|W}}, 
$$ 
where $ \kappa^{Z}_{j}(T)$ is a class in $H^{2l-2j}_{c}(Z,\R)$, called {\it the {tangent} class of dimension j}.

\begin{definition}\rm
    The maximal $j$ such that $ \kappa^{Z}_{j}(T)\neq 0$ is called {\it the horizontal dimension} (or {\it h-dimension} for short) of $T$ along $Z$. When $\kappa^{Z}(T)=0$, the h-dimension is $\max (l-p,0)$ by convention. 
\end{definition}

Equivalently, given a {tangent} current $\mathcal R$, the h-dimension of $T$ also equals the maximal $j$ such that
\[
    \mathcal R\wedge \pi^{*}(\omega^{j}|_{Z}) \neq 0.
\]
This is independent of the choice of $\mathcal R$. We mainly focus on the case when {the} h-dimension is minimal, i.e., {equal to} $\max (l-p,0)$. If $p\leq l$ and {the} h-dimension is minimal, there is a  positive closed $(p,p)$-current $\shad(\mathcal R)$ called the {\it shadow} of $\mathcal R$ on $Z$ such that $\mathcal R=\pi^*(\shad(\mathcal R))$.

Let $S$ be a positive closed $(q,q)$-current on $W$. Assume the intersection of $\supp (T)$ and $\supp(S)$ is compact. We define the \textit{density currents} between $T$ and $S$ as the tangent currents of $T\otimes S$ along the diagonal $\Delta$ in $W\times W$. We also define the density classes between $T$ and $S$ by $\kappa(T,S):=\kappa^\Delta(T\otimes S)$ and $\kappa_j(T,S):=\kappa_j^\Delta(T\otimes S)$. They are symmetric with {respect to} $T$ and $S$. Identifying $\Delta$ with $W$, we regard $\kappa_j(T,S)$ a class in $H^{2k-2j}_c(W,\R)$. {By \cite[Lemma 5.4]{DS18JAG}}, we have $\kappa_j(T,[Z])$ equals the canonical image of $\kappa_j^Z(T)$ in $H_c^{2k-2j}(W,\R)$. {Similarly, the \textit{h-dimension} between $T$ and $S$ is the maximal $j$ such that $\kappa_j(T,S)\neq 0$.}

We {now} consider the case when $p+q\leq k$. {For $T$ and $S$ as above,} if {the} h-dimension is minimal and the {density} current $\mathcal R$ is unique, we define $T\curlywedge S:=\shad(\mathcal R)$, and say that $T\curlywedge S$ is {\it well-defined.} For clarity, this is referred {to} as {the} intersection {\it in the sense of Dinh-Sibony}.

\subsection{Intersection given by super-potentials}
    Let $X$ be a compact K\"ahler manifold of dimension $m$. {Recall that $M''\Subset M'\Subset M\subset \C^p$ and $N''\Subset N'\Subset N\subset \C^{k-p}$ are all bounded convex domains, and $D=M\times N$ and $D'=M'\times N'$.} Let $T_v$ be a positive closed vertical current of bi-dimension $(k-p,k-p)$ in $M''\times N$ with continuous super-potentials. Suppose $\widehat{T}_h$ is a positive closed current of bi-dimension $(p,p)$ in $M\times N''\times X$. Then we can define an intersection $\widehat{T}_h\wedge\pi_D^*(T_v)$  by
\begin{equation}\label{eq:sp}
    \lp\widehat{T}_h\wedge\pi_D^*(T_v), \varphi\rp:=\lp T_v, (\pi_D)_*(\varphi\widehat{T}_h)\rp
\end{equation}
    for any test function $\varphi$ on $D\times X$. The right hand side is well-defined since we can check that $(\pi_D)_*(\varphi\widehat{T}_h)\in\DSH_h(D')$ and $T_v$ has continuous super-potentials. Note that the intersection also enjoys some continuity property: suppose $\{\widehat{T}_{h,n}\}$ is a sequence of positive closed currents of bi-dimension $(p,p)$ in $M\times N''\times X$ converging to $\widehat{T}_h$ as $n\to\infty$. Then we have 
    \begin{equation}\label{eq:continuous-sp}
        \widehat{T}_h\wedge\pi_D^*(T_v)=\lim_{n\to\infty}\widehat{T}_{h,n}\wedge\pi_D^*(T_v).
    \end{equation}
    This {also} follows from the fact that for any fixed test function $\varphi$, $\{(\pi_D)_*(\varphi\widehat{T}_{h,n})\}$ is a bounded family in $\DSH_h(D')$ and the super-potentials of $T_v$ are continuous. 

\begin{theorem}\label{thm:sp-density}
     Let $X $, $ T_v $, and $ \widehat{T}_h$ be as above. Then the intersection $\widehat{T}_h\curlywedge\pi_D^*(T_v)$ is well-defined and equal to $\widehat{T}_h\wedge\pi_D^*(T_v)$ {as} defined in (\ref{eq:sp}). {In particular, $\kappa_r(\widehat{T}_h,\pi_D^*(T_v))=0$ for all $r>0$. }
\end{theorem}

When $m=0$, i.e., $X$ is a point, an immediate corollary is that if $T_h$ is a positive closed horizontal current of bi-dimension $(p,p)$ in $M\times N''$, the intersection $T_h\curlywedge T_v$ is well-defined and equal to $T_h\wedge T_v$. This is proven by Dinh-Nguyen-Vu in \cite{DNV} in the compact K\"ahler setting. 

{Instead} of proving the theorem directly, we will first prove the following result. Let $\Delta$ be the diagonal of $D\times D$. We use the coordinates $(x,y)$ for a point in $D\times D$. We also use the coordinate $(x,z)$ where $z:=y-x$. Then {we have} $\Delta=\{z=0\}$. Let $\pi_{\Delta\times X}:\mathbb E\to \Delta\times X$ be the normal vector bundle of $\Delta\times X$. We identify $\Delta\times X$ with the zero section of $\mathbb E$. Then $\mathbb E$ can be identified with $\Delta\times X\times \C^k$. The dilation map $A_\lambda:\mathbb E\to\mathbb E$ is given by $A_\lambda(x,\xi,z)=(x,\xi,\lambda z)$. The projection $\pi_{\Delta\times X}$ is given by $(x,\xi,z)\mapsto (x,\xi,0)$. Let $\iota:D\times X\to \Delta\times X$ be the canonical identification $(x,\xi)\mapsto ((x,x),\xi)$. 

\begin{proposition}\label{equidensity}
    The tangent current of $\widehat{T}_h\otimes T_v$ along $\Delta\times X$ is unique and given by $\pi_{\Delta\times X}^*(\iota_*(\widehat{T}_h\wedge\pi_D^*(T_v)))$.
\end{proposition}

{For the proof, }let $\mathcal{R}$ be a limit current of $(A_\lambda)_*(\widehat{T}_h\otimes T_v)$ on $\Delta\times X\times \C^k$ as $\lambda \to \infty$. Denote the measure $\widehat{T}_h\wedge\pi_D^*(T_v)$ by {$\nu$} for simplicity. We need to {show} that {we have} $\mathcal R=\pi_{\Delta\times X}^*(\iota_*(\nu))$. We identify $D\times X\times D$ with an open subset of $\Delta\times X\times \C^k$. Let $\gamma$ be a smooth $(k,k)$-form compactly supported in $\Delta\times X\times \C^k$. Since both $\mathcal R$ and $\pi_{\Delta\times X}^*(\iota_*(\nu))$ are invariant under the action of $A_\lambda$, we may assume $\gamma$ is supported in $D\times X\times D$ by replacing $\gamma$ with $A_{\lambda_0}^*(\gamma)$ if necessary with some large enough $\lambda_0>0$. {Denote by $\pi_{D\times X}:(D\times X)\times D\to D\times X$ and $\pi_D:(D\times X)\times D\to D$ the projection maps. } {We} have
\begin{equation}\label{eq:Phi-lambda}
    \begin{aligned}
        \lp \mathcal R,\gamma\rp&=\lim_{n\to\infty}\lp (A_{\lambda_n})_{*}(\widehat{T}_h\otimes T_v), \gamma\rp \\
        &= \lim_{n\to\infty}\lp \widehat{T}_h\otimes T_v, A_{\lambda_n}^*(\gamma)\rp = \lim_{n\to\infty}\lp T_v,\Phi_{\lambda_n}\rp.
    \end{aligned}
\end{equation}
where $\Phi_{\lambda}=(\pi_{D})_*\big(A_{\lambda}^*(\gamma)\wedge \pi_{D\times X}^*(\widehat{T}_h)\big)$ for each $\lambda\in\C^*$. {Although we are in the non-compact setting, the proofs of \cite[Lemmas 3.1 \& 3.2]{DNV} still work. It follows that} $\Phi_{\lambda}$ is a smooth $(k-p,k-p)$-form on $D$ {and}
\begin{equation}\label{eq:smooth-weak}
        \lim_{\lambda\to\infty}\Phi_{\lambda}\ =\  (\pi_{D})_*\big(\iota^*((\pi_{\Delta\times X})_*(\gamma))\wedge \widehat{T}_h\big)
\end{equation} 
in {the sense of currents}. When $\lambda$ goes to infinity, the support of $\Phi_\lambda$ converges to a subset of $\supp((\pi_D)_*\widehat{T}_h)$. Therefore, in the rest of the proof, we fix a $\delta>0$ such that when $|\lambda|>1/\delta$, $\supp(\Phi_\lambda)$ is horizontal and thus $\lp T_v,\Phi_\lambda\rp$ is well-defined.

 By Weierstrass approximation theorem, {a general $\gamma$ can be uniformly approximated by a linear combination of forms of the form $\alpha(x,z)\wedge\beta(\xi)$. Here $\alpha(x,z)$ is a smooth $(t,t')$-form on $D \times D$ with compact support and $\beta(\xi)$ is a smooth $(k-t,k-t')$-form on $X$ with $0\leq t,t'\leq k$. Therefore, by linearity and continuity, to prove \cref{equidensity}, it is enough to assume $\gamma=\alpha(x,z)\wedge\beta(\xi)$ and prove that }
 \[
     \lp \mathcal{R},\gamma\rp=\lp \pi_{\Delta\times X}^*(\iota_*(\nu)), \gamma\rp.
 \]
{Depending on the choice of $(t,t')$, the proof will be divided into two cases:} 
\begin{itemize}
    \item Case 1: $(t,t')\neq (k,k)$;
    \item Case 2: $(t,t')=(k,k)$. 
\end{itemize}

{Without loss of generality, we may further assume $\alpha$ is supported in a small neighborhood of $\Delta\subset D\times D$. Note that in Case 2, $\beta$ is a smooth function on $X$.}
\medskip

\textbf{\textit{Case 1:}} $\gamma=\alpha(x,z)\wedge \beta(\xi)$.
\medskip

{For bi-degree reason, in this case we have} $(\pi_{\Delta\times X})_*(\alpha\wedge\beta)=0$. {Therefore,} we need to show that $\lp \mathcal R,\alpha\wedge\beta\rp=0.$ This is done in the following lemma. 

\begin{lemma}\label{lemma:t<k}
    When $(t,t')\neq (k,k)$, we have $\lp \mathcal R, \alpha\wedge\beta\rp=0.$
\end{lemma}

\begin{proof}
    Suppose the lemma is true for $t=t'$. Now if $t\neq t'$, suppose $t<t'$ without loss of generality. By linearity, we may assume $\alpha=\alpha_t\wedge g$, where $\alpha_t$ is a positive $(t,t)$-form {with} compact support on $D\times D$ and $g$ is a $(0,t'-t)$-form. Similarly, we take $\beta=\beta_{k-t'}\wedge g'$ where $\beta_{k-t'}$ is a positive $(k-t',k-t')$-form and $g'$ is a $(t'-t,0)$-form on $X$. Denote by $S_\lambda:=(A_{\lambda})_{*}(\widehat{T}_h\otimes T_v)$. By Cauchy-Schwarz inequality, {we have}
    \[
        |\lp S_\lambda, \alpha\wedge\beta\rp|^2\leq |\lp S_\lambda, \alpha_t\wedge g\wedge \bar{g}\wedge \beta_{k-t'}\rp|\cdot |\lp S_\lambda, \alpha_t\wedge \beta_{k-t'}\wedge g'\wedge \bar{g}'\rp|.
    \]
    Since $\|S_\lambda\|$ is bounded from above by a constant independent of $\lambda$, so is the first term on the right hand side. Now $\beta_{k-t'}\wedge g'\wedge\bar{g}'$ is a $(k-t,k-t)$-form. It follows from our assumption that $\lp S_\lambda, \alpha\wedge\beta\rp$ converges to 0.

    Therefore, we can assume $\alpha$ is a positive $(t,t)$-form on $D\times D$ and $\beta$ is a positive $(k-t,k-t)$-form on $X$ for some $t<k$. Recall that $\omega_X$ is the K\"ahler form of $X$. We can find $C>0$ such that $\beta\leq C\omega_X^{k-t}$. Hence we replace $\beta$ by $\omega_X^{k-t}$ in the {rest of the} proof. 

    In (\ref{eq:Phi-lambda}), we have
    \begin{equation}\label{eq:Phi-lambda-<k}
        \Phi_\lambda=(\pi_{D})_*(A_{\lambda}^*(\alpha)\wedge\omega_X^{k-t}\wedge \pi_{D\times X}^*(\widehat{T}_h)).
    \end{equation}
    When $t<k$, {we have} $(\pi_{\Delta\times X})_*(\alpha\wedge\omega_X^{k-t})=0$. Therefore, similar to (\ref{eq:smooth-weak}), $\Phi_\lambda$ converges weakly to 0. By \cref{lemma:Phi-lambda-bounded} {below}, we can apply the fact that $T_v$ has continuous super-potentials and deduce from (\ref{eq:Phi-lambda}) that $\lp \mathcal R,\alpha\wedge \beta\rp=0.$
\end{proof}

\begin{lemma}\label{lemma:Phi-lambda-bounded}
    For $\Phi_\lambda$ defined in (\ref{eq:Phi-lambda-<k}), $\|\ddc\Phi_\lambda\|_*$ is bounded from above by some constant independent of $\lambda$.
\end{lemma}

{In order not to interrupt the flow of the proof, we postpone the proof of \cref{lemma:Phi-lambda-bounded} to the end of this section and proceed to consider Case 2.}
\medskip

\textbf{\textit{Case 2:}} $\gamma=\alpha(x,z) \cdot \beta(\xi)$.
\medskip

    Here, $\beta$ is a smooth function on $X$ and $\alpha$ is a smooth $(k,k)$-form on $D\times D$ supported in a small neighborhood of $\Delta$. {Recall that we use $\pi: (x,y)\mapsto z=y-x$ as the projection along $\Delta$. For any $k\times k$ matrix $A$, we define the perturbed projection $\pi_A: (x,y)\mapsto z+A\cdot y$. Let $\Omega_k$ be a positive $(k,k)$-form on $\C^k$ supported near 0. As in \cref{example:ver-strict-posi}, linear combination of $\pi_{A}^*(\Omega_k)$ for different $A$'s with small coefficients produces a positive closed $(k,k)$-form on $D\times D$ which is strictly positive in a neighborhood of $\Delta$. Then, similar to Step 1 in the proof of \cref{prop:non-closed-T^+},} we may directly assume $\gamma=h\cdot\alpha_0$ where $h$ is a smooth function on $D \times X\times D$ with compact support {and $\alpha_0$ is positive closed smooth $(k,k)$-form on $D\times D$ supported near the diagonal $\Delta$}. {For such $\gamma$, the form $\Phi_\lambda$ defined in \eqref{eq:Phi-lambda} takes the form}
    \begin{equation}\label{eq:Phi-lambda=k}
        \Phi_{\lambda}=(\pi_{D})_*\big( A_{\lambda}^*(h\cdot\alpha_0)\wedge \pi_{D\times X}^*(\widehat{T}_h)\big)
    \end{equation}
    for every $|\lambda|>1/\delta$. {In parallel with \cref{lemma:Phi-lambda-bounded}, we have the following result.}

    \begin{lemma}\label{lemma:Phi-lambda-bounded-2}
    For $\Phi_\lambda$ defined in \eqref{eq:Phi-lambda=k}, $\|\ddc\Phi_\lambda\|_*$ is bounded from above by some constant independent of $\lambda$.
    \end{lemma}

    Since the proof of \cref{lemma:Phi-lambda-bounded-2} is similar to that of \cref{lemma:Phi-lambda-bounded} above, we also postpone it to the end of this section. Assuming \cref{lemma:Phi-lambda-bounded-2}, it follows from the continuity of super-potentials of $T_v$, (\ref{eq:Phi-lambda}), {and} (\ref{eq:smooth-weak}) that
   \begin{align*}
           \lp \mathcal R,\gamma\rp
           &=\lim_{n\to\infty}\lp T_v,\Phi_{\lambda_n}\rp\\
           &=\lp T_v, (\pi_{D})_*\big(\iota^*((\pi_{\Delta\times X})_*(\gamma))\wedge \widehat{T}_h\big)\rp
           =\lp \pi_{\Delta\times X}^*(\iota_*(\pi_D^*(T_v)\wedge\widehat{T}_h)),\gamma\rp.
   \end{align*}
   
{This concludes the proof for Case 2, and hence the proof of \cref{equidensity}.}

\begin{proof}[Proof of \cref{thm:sp-density}]
    We show that \cref{equidensity} implies \cref{thm:sp-density} and thus finish the proof.  Let $\Delta_X$ be the diagonal of $(D\times X)\times (D\times X)$. Let $(W,\xi)$ be a small local chart of $X$. {Using a} partition of unity, it suffices to work in $(D \times W) \times (D \times W) $ with local coordinate $(x,\xi_1,y,\xi_2)$. Define $\zeta=\xi_2-\xi_1$ and $z=y-x$. We also use $(x,\xi_1,z,\zeta)$ as a local coordinate. Then {we have} $\Delta_X =\{ z=0, \zeta=0 \}$ in this local chart. We identify the normal bundle of $\Delta_X$ with $\Delta\times W\times \C^k\times \C^m$. We can identify $(D\times W)\times (D\times W)$ with an open subset of $\Delta\times W\times \C^k\times \C^m$.  The dilation map, which we still denote by $A_\lambda$, is given by  $(x,\xi_1,z,\zeta)\mapsto(x,\xi_1,\lambda z,\lambda\zeta)$. 
    
     Let $\widehat{\gamma}$ be a smooth $(k+m,k+m)$-form on $D\times W\times D\times W$ with compact support. {Similar to the calculation in \eqref{eq:Phi-lambda}}, we have 
\[
    \lp (A_\lambda)_*(\widehat{T}_h\otimes\pi_D^*(T_v)),\widehat{\gamma}\rp=\lp T_v, (\pi_D\circ{\pi}_2)_*\big(A_\lambda^*(\widehat{\gamma})\wedge {\pi}_1^*(\widehat{T}_h)\big)\rp
\]
    where $ {\pi}_1$ and ${\pi}_2$ are the projections from $(D\times W)\times (D\times W)$ to the first and second $D\times W$, respectively. Let $\pi':(D\times W)\times (D\times W)\to (D\times W)\times D$ be given by $(x,\xi_1,y,\xi_2)\mapsto (x,\xi_1,y)$. By Fubini's theorem, {we have}
    \begin{align*}
        (\pi_D\circ{\pi}_2)_*\big(A_\lambda^*(\widehat{\gamma})\wedge {\pi}_1^*(\widehat{T}_h)\big)
        &=\int_{\xi_2}\int_{(x,\xi_1)}\widehat{\gamma}(x,\xi_1,\lambda(y-x),\lambda(\xi_2-\xi_1))\wedge\widehat{T}_h(x,\xi_1)\\
        &=\int_{(x,\xi_1)}\left(\int_{\xi_2}\widehat{\gamma}(x,\xi_1,\lambda(y-x),\lambda(\xi_2-\xi_1))\right)\wedge\widehat{T}_h(x,\xi_1)\\
        &=\int_{(x,\xi_1)}\left(\int_{\xi_2}\widehat{\gamma}(x,\xi_1,\lambda(y-x),\xi_2)\right)\wedge\widehat{T}_h(x,\xi_1)\\
        &=\int_{(x,\xi_1)}\left(A_\lambda^*((\pi')_*\widehat{\gamma})\right)\wedge\widehat{T}_h(x,\xi_1)\\
        &= ({\pi}_2)_*\big(A_\lambda^*((\pi')_*\widehat{\gamma})\wedge {\pi}_1^*(\widehat{T}_h)\big).
    \end{align*}
    Here we also use $\pi_1$ and $\pi_2$ to denote the projection maps from $(D\times W)\times D$ to $D\times W$ and $D$. Therefore 
    \[
        \lp (A_\lambda)_*(\widehat{T}_h\otimes\pi_D^*(T_v)),\widehat{\gamma}\rp=\lp (A_\lambda)_*(\widehat{T}_h\otimes T_v),(\pi')_*\widehat{\gamma}\rp.
    \]
    {We let $\lambda$ go to infinity. By \cref{equidensity}, the right hand side converges to $\lp (\pi_{\Delta\times X}\circ\pi')^*(\iota_*(\widehat{T}_h\wedge\pi_D^*(T_v))), \widehat{\gamma}\rp$. After making suitable identifications, $(\pi_{\Delta\times X}\circ\pi')^*(\iota_*(\widehat{T}_h\wedge\pi_D^*(T_v)))$ is exactly the pullback of the measure $\widehat{T}_h\wedge\pi_D^*(T_v)$ along the normal bundle of $\Delta_X$. This implies the result.} 
\end{proof}

\subsection{Proofs of \cref{lemma:Phi-lambda-bounded} and \cref{lemma:Phi-lambda-bounded-2}} 

 {The difficulty to prove \cref{lemma:Phi-lambda-bounded} is as follows: now we have $\ddc\Phi_\lambda=(\pi_{D})_*\big(A_{\lambda}^*(\ddc\alpha)\wedge\omega_X^{k-t}\wedge \pi_{D\times X}^*(\widehat{T}_h)\big)$. We write $A_\lambda^*(\ddc\alpha)$ as a difference of two positive closed $(t+1,t+1)$-forms, say $A_\lambda^*(\ddc\alpha)=\Theta_{1,\lambda}-\Theta_{2,\lambda}$. To ensure that $(\pi_{D})_*\big(\Theta_{i,\lambda}\wedge\omega_X^{k-t}\wedge \pi_{D\times X}^*(\widehat{T}_h)\big)$ is still horizontal, we need $\Theta_{i,\lambda}$ is supported near the diagonal. However, due to geometric reasons, when $t<k-1$, such forms do not exist. This is also the reason why the method of \cite{DNV} does not apply in our case.} The following is a technical preparation.

\begin{lemma}\label{lemma:basic-cons}
 For any $\lambda>2$ and $1\leq t\leq k$, there exists a smooth positive $(t,t)$-form $\theta_{\lambda,t}(x,z)$ on $D\times\C^k$ such that
\begin{enumerate}
    \item[{\rm (i)}] $\|\theta_{\lambda,t}\|_{D\times\C^k}$ is bounded by some constant independent of $\lambda$;
    \item[{\rm (ii)}] $\supp (\theta_{\lambda, t}) \subset D \times \{|z|<4 \}$ and $\theta_{\lambda,t}=(\ddc\|x\|^2+\lambda^2\ddc \|z\|^2)^t$ when $\|z\|<1/\lambda;$
    \item[{\rm (iii)}] $\d\theta_{\lambda,t}$ is independent of $\lambda$.
\end{enumerate}
\end{lemma}

\begin{proof}
    For each $\lambda > 2$, we can construct an increasing convex smooth function $\phi_{\lambda}$ on $\mathbb{R}$ satisfying the following properties:
\begin{enumerate}
    \item [(1)] $\phi_{\lambda}(s)=\lambda^{2}{\rm e}^s$ when $s<-2\log\lambda$;
    \item [(2)] $\phi_{\lambda}(s)=2s+2+2\log\lambda$ when $s>0$.
\end{enumerate}
{Define $\omega_0=4\ddc(\log\|z\|)$}, and {take} $\omega_{\lambda}(z)= \ddc \phi_{\lambda}(\log \|z\|^2)$. Then $\omega_{\lambda}(z)$ is a smooth positive closed $(1,1)$-form on $\C^k$. By construction, it is easy to see that $\omega_{\lambda}=\lambda^2\ddc\|z\|^2$ when $\|z\|<1/\lambda$ and $\omega_{\lambda}=\omega_0$ when $\|z\|>1$. Let $\chi:\C^k\to[0,1]$ be a smooth cut-off function such that $\chi=1$ when $z\in \B_k(0,2)$ and $\supp(\chi)\subset \B_k(0,4)$. Define 
$$
    \theta_{\lambda,t}=\chi(z)(\ddc\|x\|^2+\omega_{\lambda}(z))^t.
$$ 
Then we can check {that the $\theta_{\lambda,t}$'s satisfy} properties (ii) and (iii). To prove (i), {since} $\theta_{\lambda,t}=\sum_{l=0}^{t}\binom{t}{l}(\ddc \| x\|^2)^{t-l}\wedge \chi\omega_{\lambda}^l(z)$, it suffices to prove $\|\chi\omega_{\lambda}^l(z)\|_{\C^k}$ is bounded independently of $\lambda$ for all $1\leq l\leq t$.
By Stokes' formula, we have
\begin{align*}
    \|\chi\omega_{\lambda}^l(z)\|_{\C^k}
    =\int_{\C^k} \chi\omega_{\lambda}^l(z)\wedge (\ddc\|z\|^2)^{k-l}
    =\int_{\C^k}\chi\omega_0^l\wedge(\ddc\|z\|^2)^{k-l}
    \leq \|\omega_0^l\|_{\B_k(0,4)}
\end{align*}
which is bounded uniformly for any $\lambda>2$. 
\end{proof}

\begin{proof}[Proof of \cref{lemma:Phi-lambda-bounded}]

    Recall that we have $\ddc\Phi_\lambda=(\pi_{D})_*\big(A_{\lambda}^*(\ddc\alpha)\wedge\omega_X^{k-t}\wedge \pi_{D\times X}^*(\widehat{T}_h)\big)$. To overcome the difficulty mentioned {before \cref{lemma:basic-cons}}, we drop the condition on closedness and write $A_\lambda^*(\ddc\alpha)$ as a difference of two positive $(t+1,t+1)$-forms supported near the diagonal. This gives two positive horizontal currents on $D$ whose difference is $\ddc\Phi_\lambda$. Then we modify these two currents to make them closed.

     For simplicity, {take} $S_0=\omega_X^{k-t}\wedge \pi_{D\times X}^*(\widehat{T}_h)$. Without loss of generality, we may assume $\supp(\alpha)\subset\{\|z\|\leq 1\}$ and $\ddc\alpha\leq (\ddc\|x\|^2+\ddc\|z\|^2)^{t+1}$. When $|\lambda|> 1/\delta$, $A_\lambda^*(\ddc\alpha)$ is supported in $\{\|z\|<1/|\lambda|\}$ and $A_\lambda^*(\ddc\alpha)\leq (\ddc\|x\|^2+|\lambda|^2\ddc\|z\|^2)^{t+1}$. Since $t<k$, by \cref{lemma:basic-cons}, we can find a smooth positive $(t+1,t+1)$-form $\Theta_{1,\lambda}$ on $D\times D$ such that when $|\lambda|>1/\delta$, we have
     \begin{equation}\label{eq:Theta-1,lambda}
         \|\Theta_{1,\lambda}\|_{D \times D}\text{ is bounded}, \Theta_{1,\lambda}\geq A_\lambda^*(\ddc\alpha),\text{ and }\d\Theta_{1,\lambda}\text{ is independent of }\lambda.
     \end{equation}
     Define $\Theta_{2,\lambda}:=\Theta_{1,\lambda}-A_\lambda^*(\ddc\alpha)$. Therefore, we have $A_\lambda^*(\ddc\alpha)=\Theta_{1,\lambda}-\Theta_{2,\lambda}$. {Take} 
     \[
         \Omega_{1,\lambda}=(\pi_{D})_*(\Theta_{1,\lambda}\wedge S_0) \quad \text{and} \quad \Omega_{2,\lambda}=(\pi_{D})_{*}(\Theta_{2,\lambda}\wedge S_0)
     \]
     Note that $\Omega_{i,\lambda}$ is well-defined and positive horizontal in $M\times N''$ when $|\lambda|> 1/\delta$, and we have $\ddc \Phi_{\lambda}=\Omega_{1,\lambda}-\Omega_{2,\lambda}$. 

    We need to modify {the} $\Omega_{i,\lambda}$'s to make them closed. Recall that $\d\Omega_{1,\lambda}$ is independent of $\lambda$ when $|\lambda|> 1/\delta$. Fix such a $\lambda_0$. Then $\Omega_{1,\lambda}-\Omega_{1,\lambda_0}$ is a smooth closed horizontal form on $M\times N''$. Choose a positive closed horizontal smooth $(k-p+1,k-p+1)$-form $\Omega_0$ on $M\times N'$ such that $\Omega_0\geq \Omega_{1,\lambda_0}$. {This is possible because of \cref{example:ver-strict-posi}.} Define 
    \[
        \Psi_{1,\lambda}:=\Omega_{1,\lambda}-\Omega_{1,\lambda_0}+\Omega_0 \quad\text{and}\quad \Psi_{2,\lambda}:=\Omega_{2,\lambda}-\Omega_{1,\lambda_0}+\Omega_0.
    \]
    {By our choice, } they are both positive closed horizontal $(k-p+1,k-p+1)$-forms in $M\times N'$.
    
    \medskip
    \noindent {\textbf{\textit{Claim:}} both $\|\Omega_{1,\lambda}\|_{D'}$ and $\|\Omega_{2,\lambda}\|_{D'}$ are bounded by some constant independent of $\lambda$.}
    \medskip

    Assuming the claim, we deduce that the masses of $\Psi_{1,\lambda}$ and $\Psi_{2,\lambda}$ on $D'$ are bounded independently of $\lambda$. Since $\ddc\Phi_\lambda=\Psi_{1,\lambda}-\Psi_{2,\lambda}$, this implies the result of \cref{lemma:Phi-lambda-bounded}.
    
    {Now we prove the claim. Let $\Upsilon$ be a positive smooth $(p-1,p-1)$-form supported on $D'$ whose coefficients are smaller than or equal to 1. Then {we deduce that} $$\lp \Omega_{1,\lambda},\Upsilon\rp=\lp\widehat{T}_h,(\pi_{D\times X})_*(\pi_D^*(\Upsilon)\wedge \Theta_{1,\lambda}\wedge \omega_X^{k-t})\rp=\lp\widehat{T}_h\wedge \omega^{k-t}_{X}, (\pi_{D\times X})_*(\pi_D^*(\Upsilon)\wedge \Theta_{1,\lambda})  \rp.$$ By the first assertion of \eqref{eq:Theta-1,lambda}, $(\pi_{D\times X})_*(\pi_D^*(\Upsilon)\wedge \Theta_{1,\lambda})$ is a smooth positive form whose $\mathcal C^0$-norm is bounded independently of $\lambda$. This implies the claim for $\Omega_{1,\lambda}$. The proof for $\Omega_{2,\lambda}$ is similar.}
\end{proof}

    \begin{proof}[Proof of \cref{lemma:Phi-lambda-bounded-2}]
        In this case $\ddc\Phi_\lambda=(\pi_{D})_*\big( A_{\lambda}^*(\ddc h\wedge\alpha_0)\wedge \pi_{D\times X}^*(\widehat{T}_h)\big)$. Let $\omega$ be the K\"ahler form of $D\times X\times D$. We have $|\ddc h|\leq C \omega$ for some $C>0$. Recall that this means $-C\omega\leq\ddc h\leq C \omega$. Define 
    \begin{equation}\label{eq:Psi-type-2}
        \Psi_\lambda:=(\pi_{D})_*\big( A_{\lambda}^*(\omega\wedge\alpha_0)\wedge \pi_{D\times X}^*(\widehat{T}_h)\big).
    \end{equation}
    
\noindent {\textbf{\textit{Claim:}} the mass of $\Psi_\lambda$ on $D'$ is bounded independently of $\lambda$ for $|\lambda|>1/\delta$.}
\medskip
    
   It is easy to see that $|\ddc\Phi_\lambda|\leq C\Psi_\lambda$ on $D'$ and $\Psi_\lambda$ is positive closed horizontal on $D'$. The result then follows. {It remains to prove the claim.} Let $\varphi$ be a positive smooth $(p-1,p-1)$-form supported on $D'$ whose coefficients are smaller {than} or equal to 1. Then {we have} $\lp \Psi_\lambda,\varphi\rp=\lp\widehat{T}_h,(\pi_{D\times X})_*(\pi_D^*(\varphi)\wedge A_\lambda^*(\omega\wedge\alpha_0))\rp$. It is easy to compute that $(\pi_{D\times X})_*(\pi_D^*(\varphi)\wedge A_\lambda^*(\omega\wedge\alpha_0))$ is a smooth form whose $\mathcal C^0$-norm is bounded independently of $\lambda$. {This implies the claim.}
\end{proof}

\section{Stable and unstable Oseledec measures}\label{sec:oseledec}

In this section, we prove a convergence theorem for measures on $D\times \G(p,k)$ (\cref{thm:unstable-Oseledec}). Prior to that, we construct the so-called {\it (un)stable Oseledec measures}. They are basically lifts of the equilibrium measure $\mu$ {of $f$} to the Grassmannian bundles. We also recall some basic Pesin theory used in our proof.

\subsection{Oseledec measures and Pesin sets}
As usual, we consider a H\'enon-like map $f$ satisfying \massump. It is shown in \cite{DNS,DS06AIF} that the equilibrium measure $\mu$ of $f$ is mixing and hyperbolic with $p$ strictly positive and $k-p$ strictly negative Lyapunov exponents. Therefore by Oseledec's theorem, for $\mu$-almost every $x\in D$, we have a decomposition $T_x\C^k=E_s(x)\oplus E_u(x)$, where $E_s(x)$ and $E_u(x)$ are subspaces of dimensions {$k-p$ and $p$}, respectively. They are invariant in the sense that $\d f_{x}\big(E_s(x)\big)=E_{s}\big(f(x)\big)$ and $\d f_{x}\big(E_u(x)\big)=E_{u}\big(f(x)\big)$. They are called the {\it stable} and {\it unstable tangent subspaces} at $x$. 

Let $[E_s(x)]$ be the point in $\G(k-p,k)$ corresponding to the subspace $E_s(x)$. The set of points $(x,[E_s(x)])$, {when defined,} in  $D\times \G(k-p,k)$ can be seen as a measurable graph over $\mu$-almost every point in the support of $\mu$. This allows us to lift $\mu$ to a probability measure $\mu_+$ on this graph which is called the {\it stable Oseledec measure}  associated with $f$ and $\mu$. Equivalently, it has the disintegration: 
\[
    \mu_+=\int \delta_{[E_s(x)]}\d\mu(x),
\]
where $\delta_{[E_s(x)]}$ is the Dirac measure on $\pi_D^{-1}(x)$ at the point $[E_s(x)]$. {Recall from \cref{sec:tame} that $f$ induces an automorphism on the Grassmannian bundle $D\times \G(p,k)$, i.e., its canonical lift $\widehat{f}$. The invariance of {the} stable bundle implies that $\widehat{f}(x,[E_s(x)])=(f(x),[E_s(f(x))])$, whenever $T_x\C^k$ has the splitting.} {Hence, $\mu_+$ is $\widehat{f}$-invariant.} In the same way, we can construct the {\it unstable Oseledec measure} $\mu_-$ associated with $f$ and $\mu$. It is an {$\widehat{f}$-invariant} probability measure on the set of points $(x,[E_u(x)])$, {when defined,} in $D\times \G(p,k)$.

In general, the subspaces $E_s(x)$ and $E_u(x)$ depend on $x$ only in a measurable way. Nevertheless, in the spirit of Lusin's theorem, we can find a family of closed subsets on which $E_s(x)$ and $E_u(x)$ {depend} continuously with respect to $x$. More specifically, in the present situation, for any $\vep>0$ we can find a closed subset $\Lambda_\vep\subset D$ with $\mu(\Lambda_\vep)>1-\vep$ satisfying the following properties:
\begin{enumerate}
    \item The splitting $T_x\C^k=E_s(x)\oplus E_u(x)$ varies continuously on $\Lambda_\vep$;
    \item There exist constants $C_0>0$ and $0<\lambda_0<1$ such that for any $x\in \Lambda_{\vep}$, $v\in E_s(x)$, $w\in E_u(x)$ and all $l\in\N$,
    \begin{equation*}
        \|\d f_x^l(v)\|\leq C_0\lambda_0^l\|v\|\quad \text{and}\quad \|\d f^{-l}_x(w)\|\leq C_0\lambda_0^l\|w\|;
    \end{equation*}
\end{enumerate}

These sets are referred to as {\it Pesin sets}. There are numerous works on this topic. We refer the readers to \cite{Katok,KH,Pesin77RMS,Pesin78USSR} and the references therein for more details.

For every such $x$, denote by $H_x$ the set of points $(x,v)$ in $\pi_D^{-1}(x)\cong \G(p,k)$ such that $v$ is not transversal to $[E_s(x)]$. This is a hypersurface in $\pi_D^{-1}(x)$. Since $\G(p,k)$ is compact, we can endow $\G(p,k)$ with any smooth metric denoted by $\dist_\G$. From the above properties (1) and (2), {one can} obtain the following: 
\begin{enumerate}
    \item Both $[E_u(x)]$ and $H_x$ vary continuously on $\Lambda_\vep$;
    \item There exists $0<\lambda<1$ such that for any $\delta >0$ we can find a constant $C=C(\delta)>0$ such that for every $x\in\Lambda_\vep, l\in \N$ and any $v\in\G(p,k)$ satisfying $\dist_\G(v,H_x)>\delta$, we have 
    \begin{equation}\label{eq:hyperbolic}
            \dist_\G(\d f^l_x(v), [E_u(f^l(x))])\leq C\lambda^l\dist_\G(v,[E_u(x)]).
        \end{equation}
\end{enumerate}

\subsection{Convergence to $\mu_\pm$}
We only treat the case of $\mu_-$, but all the results hold for $\mu_+$ {up to replacing forward iterations and unstable directions with backward iterations and stable directions}. For simplicity, we use $\G:=\G(p,k)$ and {denote $m=\dim_\C\G$}. Let $\pi_{D}: D\times \G \to D$ and $\pi_\G: D\times \G \to \G$ denote the canonical projections. We use $\omega_D$ and $\omega_\G$ for the K\"ahler forms on $D$ and $\G$, respectively and $\omega:=\pi_D^*(\omega_D)+\pi_\G^*(\omega_\G)$ is the K\"ahler form of $D\times \G$. Here is the main result of this section:

\begin{theorem}\label{thm:unstable-Oseledec}
Let $f$ be a H\'enon-like map on $D$ satisfying \massump. Let $\widehat{\gamma}$, {which is not necessarily closed,} be a positive smooth $(k-p+m,k-p+m)$-form on $D\times \G(p,k)$. If $\supp(\widehat{\gamma}) \subset M\times N'\times \G(p,k)$ and $\lp T^+, (\pi_D)_*(\widehat{\gamma}) \rp =1$, then $d^{-n}(\widehat{f}^n)_*(\widehat{\gamma})\wedge \pi_D^*(T^+)$ converges to $\mu_-$ as $n\to \infty$. 
\end{theorem}

 The above theorem {is} a generalization of \cite[Proposition 4.3]{DS16} and characterizes the stable and unstable Oseledec measures. {Since the Green currents for general H\'enon-like maps are not necessarily wedge products of positive closed $(1,1)$-currents, the induction method in the proof of \cite[Proposition 4.3]{DS16} {no longer applies. Instead, we make a more systematic use of Pesin theory and combine it with pluripotential-theoretic arguments.}} In the rest of this section, we always assume $f$ satisfies the condition of this theorem.
 
 Define 
 \[
     \nu_n:=d^{-n}(\widehat{f}^n)_*(\widehat{\gamma})\wedge \pi_D^*(T^+).
 \]
 Notice that $\nu_n$ is supported in $D'\times
 \G$. We have $\|\nu_n\|=\lp \nu_n,1\rp=\lp T^+,(\pi_D)_*\widehat{\gamma}\rp=1$ for all $n$. Let $\nu$ be a limit measure of $\nu_n$, say $\nu=\lim_{j\to\infty}\nu_{n_j}$ for a subsequence $\{n_j\}$.  Choose any test function $\varphi$ on $D$. Then, by applying \cref{prop:non-closed-mu} to $f^{-1}$, we obtain
\begin{equation}\label{eq:push-forward-is-mu}
    \lp(\pi_D)_*(\nu),\varphi\rp=\lim_{j\to\infty}\lp d^{-n_j}(f^{n_j})_*((\pi_D)_*\widehat{\gamma})\wedge T^+,\varphi\rp=\lp\mu,\varphi\rp.
\end{equation}
Therefore, we conclude that $\nu$ is a probability measure supported on $D'\times\G$ such that $(\pi_D)_*\nu=\mu$.

{Thus, it remains to analyze the behavior of the limiting measure along the fibers of $\pi_D$.} Our starting point is the following lemma which replaces $\widehat{\gamma}$ by a {specific} singular current. Let $\Omega_\G$ be a smooth volume form on $\G$ with $\int_\G\Omega_\G=1$.  To simplify the {notations}, here (and afterwards) for a current $T$ on $D$ and a current $S$ on $\G$, we simply write $T\wedge S$ to denote $\pi_D^*(T)\wedge \pi_\G^*(S)$.

\begin{lemma}\label{lemma_Oseledec}
Let $f$ be as in \cref{thm:unstable-Oseledec}. Then $d^{-n} (\widehat f^n)_{*}( T^-\wedge \Omega_\G )\wedge \pi_{D}^*(T^{+})$ converges to $\mu_-$ as $n$ goes to infinity.
\end{lemma}

\begin{proof}
    Notice that $d^{-n} (\widehat f^n)_*( T^-\wedge \Omega_\G )\wedge \pi_D^*(T^+)=(\widehat{f}^n)_*(\Omega_\G)\wedge \pi_D^*(\mu)$. {We disintegrate the measure $ (\widehat{f}^n)_*(\Omega_\G)\wedge \pi_D^*(\mu)$ with respect to the measure $\mu$ on $D$. As $\mu$ is $f$-invariant, it suffices to analyse {the asymptotic behaviour of the conditional measures on the fibres of $\pi_D$} when $n$ goes to $\infty$. The hyperbolicity of $\mu$ implies that, {for $\mu$-almost every $x$}, the masses of {these conditional} measures will accumulate near {the} points $(x,[E_u(x)])$}. {Therefore the limit measure is precisely $\mu_-$. The proof is identical to the case of regular automorphisms up to necessary modifications. We refer the reader to \cite[ Lemma 4.4]{DS16} for the details.}
\end{proof}

{Recall that by \cref{lemma:convolution-T}, we can find a smooth positive closed $(k-p,k-p)$-form $\alpha_0$ and a negative $L^1$-form $U$, both defined on $M'\times N$ and supported in $D'$, such that $T^-=\alpha_0+\ddc U$ on $M'\times N$ and $\alpha_0$ is strictly positive in a neighborhood of $\supp(T^-)$. Our next step is to prove \cref{thm:unstable-Oseledec} for $\widehat{\gamma}=\alpha_0\wedge \Omega_\G$.}

\begin{proposition}\label{partunstable}
 Let $f$ and $\alpha_0$ be as above. Then $d^{-n}(\widehat{f}^{n})_{*}(\alpha_0\wedge \Omega_\G)\wedge \pi_{D}^*(T^+)$ converges to $\mu_-$ as $n$ goes to infinity.
    
\end{proposition}

The proof will {consist of} a series of lemmas. Following the previous notation, {define} 
\begin{equation}\label{eq:nu-n}
\begin{aligned}
    \nu_n &:= d^{-n}(\widehat{f}^n)_{*}(\alpha_0\wedge \Omega_\G)\wedge \pi_D^*(T^+),\\
    \nu_n' &:= d^{-n}(\widehat{f}^{n})_{*}(T^-\wedge \Omega_\G)\wedge \pi_D^*(T^+).
\end{aligned}
\end{equation}
Let $\{T^+_\theta\}_{\theta\in (0,1]}$ be a family of positive closed vertical smooth $(p,p)$-forms in $M'\times N$ {such that} $T^+_\theta$ converges to $T^+$ as $\theta$ goes to 0. Define
\[
    S_{n,\theta}:=d^{-n}(\widehat{f}^{n})_{*}(-U\wedge \Omega_\G)\wedge \pi_{D}^{*}(T^+_\theta)
\]
which is a positive current of bi-dimension $(1,1)$ with support in $D'\times \G$.
    \begin{lemma}\label{lemma:bounded-mass-S}
        The mass of $S_{n,\theta}$ is uniformly bounded for all $n$ and $\theta$.
    \end{lemma}
    \begin{proof}
        By direct computation, we have
        \[
            \|S_{n,\theta}\|=\lp S_{n,\theta},\omega\rp=\lp -U, V_{n,\theta}\rp
        \]
        where $V_{n,\theta}=(\pi_D)_*\big(\Omega_\G\wedge d^{-n}(\widehat{f}^n)^*(\pi_D^*(T^+_\theta)\wedge\omega)\big)$ is a vertical positive closed smooth $(p+1,p+1)$-form on $M'\times N$. \cref{prop:controlmass} implies the mass of $V_{n,\theta}$ is bounded uniformly for all $n$ and $\theta$. By \cite[Theorem 3.6]{BDR24}, we can find negative vertical $L^1$-forms $\Omega_{n,\theta}$ on $D'$ and a constant $c>0$ such that $\ddc\Omega_{n,\theta}=V_{n,\theta}$ on $D'$ and $\|\Omega_{n,\theta}\|_{D'}\leq c\|V_{n,\theta}\|$ for all $n$ and $\theta$. This implies the set $\{\Omega_{n,\theta}\}$ is a relatively compact set in $\DSH_v(D')$ with respect to the topology introduced in \cref{sec:revisit}. Since $T^-$ and $\alpha_0$ have continuous super-potentials, 
        \[
            \|S_{n,\theta}\|=\lp -U,\ddc\Omega_{n,\theta}\rp=\lp \alpha_0-T^-,\Omega_{n,\theta}\rp
        \]
        is bounded from above by some constant independent of $n$ and $\theta$.
    \end{proof}
    
    {For each $n\geq 1$, \cref{lemma:bounded-mass-S} gives a uniform bound on the masses of $S_{n,\theta}$ for $\theta\in(0,1]$. Hence, after choosing a sequence $\theta_i^{(n)}\to 0$ and passing to a subsequence, we may assume that $S_{n,\theta_i^{(n)}}$ converges to a positive current denoted by $S_n:=\lim_{i\to\infty}S_{n,\theta_i^{(n)}}$. This implies 
   \begin{align*}
       \ddc S_n&=\lim_{i\to\infty} d^{-n}(\widehat{f}^{n})_{*}(-\ddc U\wedge \Omega_\G)\wedge \pi_{D}^{*}(T^+_{\theta_i^{(n)}})\\
       &=\lim_{i\to\infty} d^{-n}(\widehat{f}^{n})_{*}(\alpha_0\wedge \Omega_\G)\wedge \pi_{D}^{*}(T^+_{\theta_i^{(n)}})-d^{-n}(\widehat{f}^{n})_{*}(T^-\wedge \Omega_\G)\wedge \pi_{D}^{*}(T^+_{\theta_i^{(n)}}).
   \end{align*}
   Recall that $T^+_{\theta_i^{(n)}}$ converges to $T^+$. Since $\alpha_0$ is smooth, the first term on the right hand side converges to $\nu_n$ as $i\to\infty$. Since $T^-$ has continuous super-potentials, the second term converges to $\nu_n'$. }Hence, we have 
    \begin{equation}\label{eq:ddcS_n}
        \ddc S_n=\nu_n-\nu_n'.
    \end{equation}

    {By \eqref{eq:nu-n} and \cref{lemma_Oseledec}, $\nu_n'$ converges to $\mu_-$ when $n$ goes to infinity.} By extracting a subsequence, we may assume $S=\lim_{j\to\infty} S_{n_j}$ and $\nu=\lim_{j\to\infty} \nu_{n_j}$. By (\ref{eq:ddcS_n}), we have
    \begin{equation*}
        \nu-\mu_-=\ddc S.
    \end{equation*}
    
    For every $l\geq 1$, $(\widehat{f}^{l})^{*}(\nu)=\lim_{j\to\infty}\nu_{n_j-l}$ is a limit measure of $\{\nu_n \}$. Suppose $\nuinfty$ is a limit measure of $\{ (\widehat{f}^{l})^{*}(\nu)\}$, then it is also a limit measure of $\{ \nu_n\}$. This is due to the fact that the space of all probability measures supported in $D'\times \G$ is a compact metric space under the weak topology. By the same argument as before, there exists a positive current $\Sinfty$ of bi-dimension $(1,1)$ which is a limit current of $\{S_n\}$ such that
   \begin{equation}\label{eq:nuinfty-mu-=ddcSinfty}
        \nuinfty-\mu_-=\ddc\Sinfty.
    \end{equation}

    Since $(\pi_D)_*\nuinfty=\mu$, we may assume the disintegration of $\nuinfty$ to be 
    \[
        \nu^{(\infty)}=\int_D\nuinfty_x\d\mu(x),
    \]
    {where for $\mu$-almost every $x\in D$, $\nu_x^{(\infty)}$ is a probability measure on $\pi_D^{-1}(x)\cong\G$.}
    \begin{lemma}\label{lemma:S-vertical-union}
        For $\mu$-almost every $x\in D$, there is a positive current $\Sinfty_x$ of bi-dimension $(1,1)$ on $\pi_D^{-1}(x)$ such that \[
              \ddc\Sinfty_x =\nuinfty_x-\delta_{[E_u(x)]}.
        \] In particular, $\Sinfty_x$ is $\C$-normal, {i.e., both $S_x^{(\infty)}$ and $\ddc S_x^{(\infty)}$ are of order 0.}
    \end{lemma}

    \begin{proof}
      {The last assertion is a direct consequence of the first and the definition of $S_n$.} Notice that the projections of $\nuinfty$ and $\mu_-$ to $D$ are both equal to $\mu$. We have $\ddc(\pi_D)_*\Sinfty=0$. Since $\omega_D=\ddc\|x\|^2$, we have $\Sinfty\wedge \pi_D^*(\omega_D)=0$. Following the proof of \cite[Lemmas 3.2 \& 3.3]{DS18JAG}, one can show that there exists a probability measure $\sigma$ supported in $D'$ such that for $\sigma$-almost every $x\in D$, {there is} a positive current {$V_x$} of bi-dimension $(1,1)$ on $\pi^{-1}_D(x)$ such that 
      $$
      \Sinfty=\int_D V_x\d \sigma(x).
      $$
      Define a positive measure $ \upsilon:=\mu + \sigma$ on $D$. Thus, both $\mu $ and $ \sigma $ are absolutely continuous with respect to $\upsilon$. By Radon-Nikodym theorem, we can find two $\upsilon$-integrable functions $ g(x)$ and $h(x)$ such that $ \mu= g(x) \upsilon $ and $\sigma = h(x) \upsilon$. It is easy to see that $h(x)+g(x)=1$ for $\upsilon$-almost every $x\in D$.  Now, we rewrite (\ref{eq:nuinfty-mu-=ddcSinfty}) as follows: 
      $$
       \int_{D} g(x)(\nuinfty_x-\delta_{[E_u(x)]})\d \upsilon(x)= \int_{D} h(x) \ddc V_{x}\cdot \d \upsilon(x). 
      $$
      {By uniqueness of the disintegration,} for $\upsilon$-almost every $x\in D$, we have that  $g(x)(\nuinfty_x-\delta_{[E_u(x)]})= h(x) \ddc V_{x}$. Note that for $\mu$-almost every $x\in D$, we have $g(x)>0$. {By defining $S_x^{(\infty)}=h(x)g(x)^{-1}V_x$ for $\mu$-almost every $x\in D$}, we complete the proof.
   \end{proof}

    \begin{lemma}\label{lemma:nuinfty_z}
        For $\mu$-almost every $x\in D$, $\nuinfty_x$ when viewed as a probability measure on $\G$, is supported on $\{[E_u(x)]\}\cup H_x$.
    \end{lemma}
    \begin{proof}
        Recall that for any $\vep>0$, we can choose a closed subset $\Lambda_\vep\subset \supp(\mu)$ and $0<\lambda=\lambda(\Lambda_{\vep})<1$ such that $\mu(\Lambda_\vep)>1-\vep$, and for any $\delta>0$, there {exists} some constant $C>0$, depending {only} on $\Lambda_\vep$ and $\delta$, such that {\eqref{eq:hyperbolic} holds for every $x\in \Lambda_\vep$,
        every $v\in\G$ with $\dist_\G(v,H_x)>\delta$, and every $l\in\N$.}
        
        Define $K_\vep:=\{(x,v):x\in \Lambda_\vep, v\in \{[E_u(x)]\}\cup H_x\}$, {which} is a closed subset of $D\times \G$. Let $\varphi$ be any smooth test function on $D\times \G$ with compact support such that $0\leq \varphi\leq 1$ and $\supp(\varphi)\subset (D\times\G)\setminus K_\vep$. By compactness, we can choose $\delta>0$ small enough such that for every $x\in \Lambda_{\vep}$, $\dist_\G (\supp(\varphi)\cap \pi_D^{-1}(x), H_x)>\delta$. Suppose the disintegration of $\nu$ is given by $\nu=\int\nu_x\d\mu(x)$. Then we have
        \[
            \lp (\widehat{f}^l)^*\nu,\varphi\rp=\int_D \lp \nu_x, (\widehat{f}^l)_*\varphi\rp \d\mu(x).
        \]
        For any $l\in \N$, we have $\mu(f^l(\Lambda_\vep))=\mu(\Lambda_\vep)>1-\vep$ since $\mu$ is $f$-invariant. Therefore $\mu\big(D\setminus (\Lambda_\vep\cap f^l(\Lambda_\vep))\big)<2\vep$. It follows that
        \[
            \lp (\widehat{f}^l)^*\nu,\varphi\rp\leq\int_{x\in\Lambda_\vep}\mathbf{1}_{\{x\in f^l(\Lambda_\vep)\}}\lp \nu_x,(\widehat{f}^l)_*\varphi\rp\d\mu+2\vep.
        \]
        {Define} $\varphi_l(x):=\mathbf{1}_{\{x\in f^l(\Lambda_\vep)\}}\lp \nu_x,(\widehat{f}^l)_*\varphi\rp$. 

        \medskip
        {\noindent\textit{\textbf{Claim.}}} For every $x\in \Lambda_\vep$, we have $\lim_{l\to\infty}\varphi_l(x)=0$.
        \medskip
        
        {Assuming the claim,} by dominated convergence theorem, we have $\lp \nuinfty,\varphi\rp\leq 2\vep$. Since $\varphi$ is arbitrary, we deduce that $\nuinfty( K_\vep)\geq 1-2\vep$. The result then follows as $\vep$ is also arbitrary.

        It remains to prove the claim. Fix $l\geq 1$ and $x\in\Lambda_\vep$. Suppose $x=f^l(w)$ for some $w\in \Lambda_\vep$, otherwise we directly have $\varphi_l(x)=0$. Define the annulus $\mathbb A_0([E_u(x)],r):=\{(x,v):0<\dist_\G(v,[E_u(x)])<r\}$. We can find $r>0$ such that for any $w\in \Lambda_\vep$, $\supp( \varphi)\cap\pi_D^{-1}(w)\subset \mathbb A_0([E_u(w)],r)$. As $\supp(\varphi)$ is disjoint {from} $\delta$-neighbourhood of $H_x$ for $x\in \Lambda_{\vep}$, by (\ref{eq:hyperbolic}) we have
        \begin{equation*}
            \begin{aligned}
                \supp ((\widehat{f}^l)_*\varphi)\cap \pi_D^{-1}(x)
                &=\widehat{f}^l\big(\supp(\varphi)\cap\pi_D^{-1}(w)\big)\\
                &\subset \widehat{f}^l\big(\mathbb A_0([E_u(w)],r)\big)\\
                &\subset \mathbb A_0([E_u(x)],C\lambda^lr).
            \end{aligned}
        \end{equation*}
        It follows that 
        \[
            |\lp \nu_x,(\widehat{f}^l)_*\varphi\rp|\leq \|\nu_x\|_{\mathbb A_0([E_u(x)], C\lambda^lr)},
        \]
        which converges to 0 as $l$ goes to infinity.
    \end{proof}
    
    \begin{proof}[End of the proof of \cref{partunstable}]
        {By \cref{lemma:S-vertical-union} {and \cref{lemma:nuinfty_z}}}, for $\mu$-almost every $x\in D$, {we have} $\ddc\Sinfty_x=0$ outside $Y:=\{[E_u(x)]\}\cup H_x$. {Let $R_x:=S_x^{(\infty)}|_{\mathbb G\setminus Y}$, and let $\widetilde R_x$
        denote its extension to $\mathbb G$ by zero.} Then {we have {$\ddc \widetilde{R}_x= \ddc \Sinfty_x =0 $} on $\G\setminus Y$.} {By \cite[Theorem 1.3]{DS07MM}}, $\ddc \widetilde{R}_x$ {is a negative current} on $\G$. This implies $\ddc  \widetilde{R}_x=0$ since {$\G$ is compact and we have $\|-\ddc\widetilde{R}_x\|=\lp -\ddc \widetilde{R}_x,1\rp=0 $.} 
        
        {Define $K_x:= S^{(\infty)}_{x}-\widetilde{R}_x$,} which is {a positive current supported on $Y$}. {By \cref{lemma:S-vertical-union} and the above, we have 
        \[
            \ddc K_x= v_x^{(\infty)}-\delta_{[E_u(x)]}.
        \]} 
        As $S^{(\infty)}_{x}$ is $\C$-normal and of bi-dimension $(1,1)$, it has no mass on the point $[E_u(x)]$ (see for example \cite{B}). {Therefore, $K_x$ is {in fact} supported on $H_x$. {Hence $\ddc K_x=\nuinfty_x-\delta_{[E_u(x)]}$ is also supported on $H_x$}. {Since $[E_u(x)] \notin H_x$ and $\nuinfty_x $ is a probability measure}, it follows that we must have $\nuinfty_x=\delta_{[E_u(x)]}$.} Since this is true for $\mu$-almost every $x\in D$, we deduce that $\nuinfty=\mu_-$ and thus $\lim_{l\to\infty}(\widehat{f}^l)^*\nu=\mu_-$. On each fibre, all the mass of $\nu_x$ outside the point $[E_u(x)]$ will concentrate on $H_x$ under the backward iteration of $\widehat{f}$, it follows that we must have $\nu=\mu_-$ at the beginning.
    \end{proof}

    We are ready to finish the proof of \cref{thm:unstable-Oseledec}. Prior to that, we prove a simple lemma which helps to reduce the problem to easier cases.

\begin{lemma}\label{lemma:dirac-measure-strong}
    Let {$\eta$} be a probability measure on $D\times \G$ such that $(\pi_D)_*\eta=\mu$. Suppose there exists some constant $c>0$ such that $\eta\leq c\mu_-$. Then {we have} $\eta=\mu_-$.
\end{lemma}
\begin{proof}
    We can disintegrate $\eta$ with respect to $\mu$: $\eta=\int \eta_x\d\mu(x)$ where $\eta_x$ is a probability measure supported on $\{x\}\times \G$. Recall that $\mu_-=\int \delta_{[E_u(x)]}\d\mu(x)$. If we test on arbitrary positive continuous functions, the assumption implies $c\delta_{[E_u(x)]}-\eta_x\geq 0$ for $\mu$-almost every $x\in D$. Hence, {for $\mu$-almost every $x\in D$, $\eta_x$ {is a constant multiple of} $\delta_{[E_u(x)]}$. Since $\eta_x$ has mass 1, we have $\eta_x=\delta_{[E_u(x)]}$.} The desired result follows.
\end{proof}

\begin{proof}[End of the proof of \cref{thm:unstable-Oseledec}]  

We consider a limit measure of $\nu_n$: $\nu=\lim_{j\to\infty}\nu_{n_j}$ for some subsequence $\{n_j\}$. According to \cref{lemma:dirac-measure-strong}, it suffices to find $c>0$ such that $\nu\leq c\mu_-$. By using local {product} coordinates, we can decompose $\widehat{\gamma}$ as the sum of two parts: $$\widehat{\gamma}=\widehat{\gamma}_0 + \widehat{\gamma}_1.$$
Here $\widehat{\gamma}_0$ is the sum of components of $\widehat{\gamma}$ which have bi-degree $(k-p,k-p)$ on $D$-directions and bi-degree $(m,m)$ on $\G$-directions, and $\widehat{\gamma}_1$ is the sum of the remaining terms {whose bi-degree on $D$-directions is strictly larger than $(k-p,k-p)$. Here and similarly in other places of this paper, the bi-degree in the $D$-directions (resp. $\G$-directions) means the bi-degree with
respect to the variables on the $D$-factor (resp. $\G$-factor) in local product coordinates.}

Notice that  $d^{-n}(\widehat{f}^n)_*(\widehat{\gamma}_1)\wedge (\pi_{D})^{*}T^+=0$. {In fact, if we take a smooth test function $\varphi$ compactly supported on $D\times\G$, we have 
\[
    \lp (\widehat{f}^n)_*(\widehat{\gamma}_1)\wedge (\pi_{D})^{*}T^+, \varphi\rp=\lp T^+, (\pi_D)_*\big(\varphi(\widehat{f}^n)_*(\widehat{\gamma}_1)\big)\rp.
\]
Then we have $(\pi_D)_*\big(\varphi(\widehat{f}^n)_*(\widehat{\gamma}_1)\big)=0$ by our choice of $\widehat{\gamma}_1$.} Therefore, we may directly take $\widehat{\gamma}=\widehat{\gamma}_0$. In this case, we can find a smooth positive horizontal $(k-p,k-p)$-form $\alpha$ on $D$ such that $\lp T^+,\alpha\rp=1$ and $\widehat{\gamma}\lesssim \alpha\wedge \Omega_\G$. Hence, it suffices to prove the result for $\widehat{\gamma}=\alpha\wedge \Omega_\G$.

Let $\alpha_0$ be as in \cref{lemma:convolution-T}. Let $\chi$ be a smooth cut-off function such that $\chi=1$ on a small neighborhood of $\supp(T^{-})$ and $\supp(\chi)\subset\supp(\alpha_0)$. We can always find such a $\chi$ since $\alpha_0$ is strictly positive on a small neighborhood of $\supp(T^-)$. Therefore, by \cref{prop:non-closed-mu},
\[
    \lim_{n\to\infty}\lp d^{-n}(f^n)_*(\alpha)\wedge T^+,\chi\rp=\lp \mu,\chi\rp=1
\]
since $\chi=1$ on $\supp(\mu)$. For any $\epsilon_0>0$, let $N$ be a positive integer such that $\lp T^+, \chi\cdot d^{-N}(f^N)_*\alpha\rp >1-\epsilon_0$. {Take} $ \alpha_{1}= \chi\cdot d^{-N}(f^N)_{*}\alpha$ and $\alpha_{2}=(1-\chi)\cdot d^{-N}(f^N)_{*}\alpha$. We define for $i=1$ and 2, a sequence of measures:
\[
    \nu_n^i:=d^{-n+N}(\widehat{f}^{n-N})_*(\alpha_i\wedge (\widehat{f}^N)_*\Omega_\G)\wedge \pi_D^*(T^+).
\]
Then {we deduce that} $\nu_n=\nu_n^1+\nu_n^2$ and $\|\nu_n^i\|\leq 1$ for all $n$ and $i$. After extracting a subsequence, we may assume $\nu_{n_j}^i$ converges to $\nu_i$ for both $i=1$ and 2 as $j$ goes to infinity.
Therefore, we have $\nu=\nu_1+\nu_2$. Moreover, by arguing in the same way as in (\ref{eq:push-forward-is-mu}), we obtain that 
\[
    (\pi_{D})_*(\nu_1)=(1-\epsilon)\mu
\]
for some $\epsilon<\epsilon_0$. Now notice that by our choice of $\chi$ and $\alpha_1$, we can find some large enough constant $c_1>0$ such that $\alpha_1\leq c_1\alpha_0$. By \cref{partunstable}, this implies $\nu_1\leq c_1\mu_-$. Applying \cref{lemma:dirac-measure-strong} {with} $\|\nu_1\|^{-1}\nu_1$ {instead of $\eta$}, we conclude that $\nu_1=(1-\epsilon)\mu_-$. Since $\epsilon_0$ is arbitrary and $\nu$ is of mass 1, we must have $\nu=\mu_-$.
\end{proof}

\section{Proof of the main theorem}\label{sec:main}

We initiate the proof of our main result. Let $f$ be a H\'enon-like map on $D$ satisfying \massump. Define $F:=(f,f^{-1})$ as a bi-holomorphic map defined on a subdomain of $D\times D$. Let $\pi_1$ and $\pi_2$ be canonical projections from $D\times D$ to the first and second factor, respectively. Since $f^{-1}$ is a H\'enon-like map after switching the role of $M$ and $N$, $F$ is also a H\'enon-like map. The following proposition  allows us to apply the results in previous sections to $F$. 

\begin{proposition}
    Let $f_1$ and $f_2$ be as in \cref{example:prod}. If both of them satisfy \massump, then so does {the product map $(f_1, f_2)$}. In particular, $F$ satisfies \massump.
\end{proposition}

\begin{proof}
    We assume $M_i\subset {\C^{p_i}}$ and $N_i\subset {\C^{k_i-p_i}}$ for some integers $k_i>p_i\geq 1$. Choose convex open sets $M_i^*$ such that $M_i''\Subset M_i^*\Subset M_i'$. Let $\{d_{l,i}^{+}\}_{l=0}^{p_i}$ and $\{d_{l,i}^{-}\}_{l=0}^{k_i-p_i}$ be the dynamical degrees of $f_i$. Finally, denote by $\{\tilde{d}_l^\pm\}$ the dynamical degrees associated with {$G:=(f_1, f_2)$}. Take $p=p_1+p_2$ and $k=k_1+k_2$. Then {we have} $\tilde{d}_q^+=\tilde{d}_{k-q}^-=d_1d_2$. For simplicity, we only prove that {we have} $\tilde{d}^+_{q-1}<d_1d_2$. The other inequality $\tilde{d}_{k-q-1}^-<d_1d_2$ can be obtained in the same way.
        
    Let $S$ be a positive closed horizontal current of bi-dimension $(q-1,q-1)$ of mass 1 in $(M_1'\times M_2')\times (N_1'\times N_2')$. Let $\omega_i$ be the K\"ahler form of $\C^{p_i}$. Let $\chi_i:M_i\to [0,1]$ be smooth cut-off functions such that $\chi_i=1$ on $M_i''$ and $\supp(\chi_i)\subset M_i^*$. After pulling back, we also regard $\chi_i$ and $\omega_i$ as objects on $D_i$. Define $\pi_{M_1\times M_2}$ {as} the projection from $D_1\times D_2$ to $M_1\times M_2$. Then it is easy to compute that
        \begin{equation}\label{eq:expand}
            \begin{aligned}
                \|(\pi_{M_1\times M_2})_*((G^n)_*S)\|_{M_1''\times M_2''}
                &\leq \lp (G^n)_*S,\chi_1\omega_1^{p_1}\otimes \chi_2\omega_2^{p_2-1}\rp \\
                &+\lp (G^n)_*S, \chi_1\omega_1^{p_1-1}\otimes\chi_2\omega_2^{p_2}\rp.
            \end{aligned}    
        \end{equation}
    We only bound the first term on the right hand side as the proof for the second term is similar. Denote by $\pi_i$ the projection map from $D_1\times D_2$ to $D_i$. {Take} 
    \[
        S_n=(\pi_{2})_*\big(S\wedge \pi_{1}^*((f_1^n)^*(\chi_1\omega_1^{p_1}))\big),
    \]
    which is a positive closed horizontal current of bi-dimension $(p_2-1,p_2-1)$ in $M_2'\times N_2'$. Then the first term is equal to $\lp (f_2^n)_*S_n,\chi_2\omega_2^{p_2-1}\rp$. 
    
    \medskip
    \noindent\textit{\textbf{Claim.}} $d_1^{-n}\|S_n\|_{M_2^*\times N_2'}$ is bounded from above by a constant independent of $S$ and $n$.
    \medskip
    
    Assuming the claim, we have 
    \[
        \lp (G^n)_*S,\chi_1\omega_1^{p_1}\otimes \chi_2\omega_2^{p_2-1}\rp\leq \lp (f_2^n)_*S_n, \chi_2\omega_2^{p_2-1}\rp\leq d_1^n \|(f_2^n)_*(d_1^{-n}S_n)\|_{M_2^*\times N_2'}.
    \]
    By our assumption on $f_2$, {we have} $\limsup_{n\to\infty}\sup_S\|(f_2^n)_*(d_1^{-n}S_n)\|^{1/n}_{M_2^*\times N_2'}\leq d_{p_2-1,2}^+$. It follows from \cref{lemma:mass-hori} and  (\ref{eq:expand}) that $\tilde{d}_{q-1}^+\leq \max(d_1d_{p_2-1,2}^+,d_2d_{p_1-1,1}^+)<d_1d_2$.

    {To prove the first assertion, it remains to prove the claim.} Choose a smooth cut-off function $\widetilde{\chi}_2:M_2\to[0,1]$ such that $\widetilde{\chi}_2=1$ on $M_2^*$ and $\supp(\widetilde{\chi}_2)\subset M_2'$. Then {we have}
    \[
        d_1^{-n}\|S_n\|_{M_2^*\times N_2'}\leq d_1^{-n}\lp S_n,\widetilde{\chi}_2\omega_2^{p_2-1}\rp=\lp d_1^{-n}(f_1^n)^*(\chi_1\omega_1^{p_1}), (\pi_{1})_*\big(S\wedge \pi_{2}^*(\widetilde{\chi}_2\omega_2^{p_2-1})\big)\rp.
    \]
    Since the $\|\cdot\|_*$-norm of $\ddc(\pi_{1})_*\big(S\wedge \pi_{2}^*(\widetilde{\chi}_2\omega_2^{p_2-1})\big)$ is bounded from above by a constant independently of $S$, \cite[Proposition 4.3]{DNS} implies the claim.

    The last assertion comes from the fact that $f^{-1}$ is also a H\'enon-like map with the same set of dynamical degrees as that of $f$.
\end{proof}

We will use $(x,y)$ for the usual complex coordinates in $\C^k\times \C^k$. Let $\Delta$ be the diagonal of $D\times D$. Let $\Gamma_n$ be the graph of $f^n$ in $D\times D$. When $n$ is even, we have $\Gamma_n=F^{-n/2}(\Delta)$. If $n$ is odd, then {we have} $\Gamma_n=F^{-(n-1)/2}(\Gamma_1)$. For {simplicity,} we assume $n$ is even. The other case can be treated similarly. Although $\Delta$ is not a vertical set, we can prove that $F^{-1}(\Delta)$ is vertical. Therefore $[\Gamma_n]=[F^{-n/2}(\Delta)]$ is a positive closed vertical current of bi-dimension $(k,k)$. Define $\mathbb T^+:=T^+\otimes T^-$ and $\mathbb T^-:=T^-\otimes T^+$. They are the Green currents for $F$, and we have $d^{-n}[\Gamma_n]\to\bbT^+$ as $n\to\infty$.

\subsection{Intersection between $\widehat{\mathbb T}^+$ and $\Pi^*[\Delta]$}
Consider the (canonical) lift of $\Gamma_n$ to $\Gr(D\times D,k)=D\times D\times \G(k,2k)$, denoted by $\widehat{\Gamma}_n$. By \cref{thm:tame}, any converging subsequence $d^{-n_j}[\widehat{\Gamma}_{n_j}]$ converges to a lift of $\mathbb T^+$. We {fix one such limit} $\widehat{\mathbb T}^+$ which depends on the choice of $\{n_j\}$. Notice that $\Gr(D,k-p)\times \Gr(D,p)$ can be naturally embedded in $\Gr(D\times D,k)$.

\begin{lemma}\label{SuppofT}
    The current $\widehat{\bbT}^+$ is supported in $\Gr(D,k-p)\times \Gr(D,p)$.
\end{lemma}
\begin{proof}[Sketch of Proof]
Since $\bbT^+\wedge \pi_1^*(\omega_D^{k-p+1})=\bbT^+\wedge\pi_2^*(\omega_D^{p+1})=0$, we can show that almost every lame of $\bbT^+$ is locally a product of a $(k-p)$-dimensional manifold in $D$ with another of dimension $p$. The result follows since the woven structure of $\widehat{\bbT}^+$ is lifted from that of $\bbT^+$. For details, {we refer to \cite[Lemma 5.6]{DS16}; since the argument there is local, it applies equally in the present setting.}
\end{proof}

\begin{remark}\rm\label{remark:woven-bbT-hat}
   Let $\pi_{\Gr_1}$ and $\pi_{\Gr_2}$ be the projections from $\Gr(D,k-p)\times \Gr(D,p)$ to the first and second factors. Denote by $\omega_{\Gr_1}$ and $\omega_{\Gr_2}$ the corresponding K\"ahler forms. Then from the above discussion, we can easily prove that {we have} $\widehat{\bbT}^+\wedge\pi_{\Gr_1}^*(\omega_{\Gr_1}^{k-p+1})=\widehat{\bbT}^+\wedge\pi_{\Gr_2}^*(\omega_{\Gr_2}^{p+1})=0$. 
\end{remark}

Let $\Pi:\Gr(D\times D,k)\to D\times D$ be the canonical projection map. Recall that for $\mu$-almost every point $z$, the stable (resp. unstable) tangent subspace is denoted by $E_s(x)$ (resp. $E_u(x)$). Define $\mu^{\Delta}:=\mathbb{T}^{+}\wedge [\Delta]$ which is exactly {equal to} $T^+\wedge T^-$ if we identify $\Delta$ with $D$. Denote by $\widehat{\mu}^{\Delta}$ the lift of $\mu^\Delta$ to the set of points $(x,x,[E_s(x)\times E_u(x)])$ in $\Gr(D\times D,k)$. The following is the main result of this subsection. 

\begin{proposition}\label{prop:intersec-mu-hat}
   The intersection $\widehat{\mathbb T}^+\curlywedge \Pi^*[\Delta]$ is well-defined and is equal to $\widehat{\mu}^\Delta$.
\end{proposition}

Recall from \cref{sec:sp} that this means the h-dimension of $\widehat{\bbT}^+$ along $\Pi^{-1}(\Delta)$ is 0 and the tangent current is given by the pullback of $\widehat{\mu}^\Delta$ to the natural compactification of the normal bundle of $\Pi^{-1}(\Delta)$ in $\Gr(D\times D,k)$.

\begin{lemma}\label{impor-h=0}
    The h-dimension of $\widehat{\mathbb T}^+$ along $\Pi^{-1}(\Delta)$ is $0$.
\end{lemma}
\begin{proof} For every $l\geq 1$, choose $\widehat{\bbT}^+_l$ to be a limit value of $ d^{-n_j+2l}[\widehat{\Gamma}_{n_j-2l}]$ when $j\to \infty$. Let $\widehat{F}$ be the canonical lift of $F$ to $\Gr(D\times D,k)$. Then we have $\widehat{\mathbb T}^+=d^{-2l}(\widehat{F}^l)^*(\widehat{\mathbb T}^+_l)$ and 
\begin{equation}\label{eq:h-dim-0}
    \kappa_r(\widehat{\mathbb T}^+,\Pi^{-1}(\Delta))=\kappa_r(\widehat{\bbT}^+_l,d^{-2l}\Pi^*[\Gamma_{-2l}])
\end{equation}
for every $r\geq 0$. Suppose $\widehat{\bbT}^+_\infty$ is a limit value of the sequence $\{\widehat{\bbT}^+_l\}$, {say $\widehat{\bbT}^+_\infty=\lim_{i\to\infty}\widehat{\bbT}^+_{l_i}$ for some subsequence $\{l_i\}$}. Recall that $d^{-2l}[\Gamma_{-2l}]$ converges to $\bbT^-$. Since $\mathbb T^-$ has continuous super-potentials, \cref{thm:sp-density} implies the intersection $\Pi^*(\mathbb T^-)\curlywedge \widehat{\mathbb T}_\infty^+$ is well-defined and thus $\kappa_r(\widehat{\mathbb T}^+_\infty, \Pi^*(\bbT^-))=0$ for any $r>0$. {Applying the upper semi-continuity of density (see \cite[Corollary 5.8]{DS18JAG}) to the sequences $\{\widehat{\bbT}_{l_i}^+\}_{i}$ and $\{d^{-2l_i}[\Gamma_{-2l_i}]\}_i$, and using \eqref{eq:h-dim-0}, we obtain} $\kappa_r(\widehat{\mathbb T}^+,\Pi^{-1}(\Delta))=0$ for any $r>0$.
\end{proof}

\begin{lemma}\label{decompcurr}
    Let $\pi_s:\Gr(D,k-p)\times \Gr(D,p)\to \Gr(D,k-p)\times D$ be the canonical projection. Define $\widehat{\mathbb T}_s^+:=(\pi_s)_*(\widehat{\mathbb T}^+)$. Then there exists a positive closed current $\widehat{T}^+$ of bi-dimension $(k-p,k-p)$ on $\Gr(D,k-p)$ such that $\widehat{\mathbb T}_s^+=\widehat{T}^+\otimes T^-$. Similarly, if we {define} $\pi_u:\Gr(D,k-p)\times \Gr(D,p)\to D\times \Gr(D,p)$ and $\widehat{\mathbb T}_u^+:=(\pi_u)_*(\widehat{\mathbb T}^+)$, then we have $\widehat{\mathbb T}_u^+=T^+\otimes \widehat{T}^-$ for some positive closed current $\widehat{T}^-$ on $\Gr(D,p)$. 
\end{lemma}

{Similar results are proved in \cite[Lemma 5.7]{DS16} for H\'enon-type maps. Our proof is different from theirs since the Green currents $T^\pm$ are no longer extremal. The notations $\widehat{T}^+$ and $\widehat{T}^-$ will be justified in the next lemma, where these currents are shown to be lifts of $T^+$ and $T^-$, respectively. }

\begin{proof}
    We only prove the first case. The proof for $\widehat{\bbT}_u^+$ is symmetric. For every $l\geq 1$, let $\widehat{\mathbb{T}}^{+}_l$ be a limit value of $d^{-n_j+2l}[\widehat{\Gamma}_{n_j-2l}]$ when $j\to \infty$. Then $\widehat{\mathbb{T}}^{+}_l$ is also a lift of $\mathbb{T}^{+}$. {Define} $\widehat{\mathbb{T}}^{+}_{s,l}:=(\pi_{s})_{*}(\widehat{\mathbb{T}}^{+}_{l})$. Recall that $\widehat{f}$ is the canonical lift of $f$ to $\Gr(D,k-p)$. Define $h:=(\widehat{f},f^{-1})$ on $\Gr(D,k-p)\times D$. Then {we have} $\widehat{\mathbb{T}}^{+}_{s}=d^{-2l}(h^l)^{*}(\widehat{\mathbb{T}}^{+}_{s,l})$. Let $\pi_\Gr$ and $\pi_D$ be the canonical projections from $\Gr(D,k-p)\times D$ to $\Gr(D,k-p)$ and $D$, respectively. Denote the K\"ahler forms of $\Gr(D,k-p)$ and $D$ by $\omega_\Gr$ and $\omega_D$.
    
    Since the projection of $\widehat{\bbT}^+_{s,l}$ onto $D\times D$ is $\bbT^+$, we have  $\supp(\widehat{\mathbb{T}}^{+}_{s,l})\subset \Gr(D,k-p)\times M\times N''$ for every $l\geq 1$. By \cref{remark:woven-bbT-hat}, we know that $\widehat{\mathbb T}_{s,l}^+\wedge \pi_\Gr^*(\omega_\Gr^{k-p+1})=0$ and $\widehat{\mathbb T}_{s,l}^+\wedge \pi_D^*(\omega_D^{p+1})=0$. Let $\alpha$ be a $(t,t')$-form on $\Gr(D,k-p)$ and $\beta$ be a $(k-t,k-t')$-form on $D$. Then it follows from the proof of \cite[Proposition 2.4]{DS16} that 
    {\begin{equation}\label{eq:vanish}
        \lp\widehat{\mathbb T}_{s,l}^+,\pi_\Gr^*(\alpha)\wedge\pi_D^*(\beta)\rp=0, \quad \text{for all } (t,t')\neq(k-p,k-p). 
    \end{equation}}
    Fix a positive $(k-p,k-p)$-form $\alpha$ on $\Gr(D,k-p)$ with compact support and a vertical $(p,p)$-form $\beta$ on $D$. Observe that 
    \begin{align*}
    \lp \widehat{\mathbb{T}}^{+}_{s},\pi_\Gr^*(\alpha)\wedge\pi_D^*(\beta)\rp 
    & = \lp d^{-2l}(h^l)^{*}(\widehat{\mathbb{T}}^{+}_{s,l}), \pi_\Gr^*(\alpha)\wedge\pi_D^*(\beta) \rp \\
    & = \lp \widehat{\mathbb{T}}^{+}_{s,l}, d^{-2l}\pi_\Gr^*\big((\widehat{f}^l)_{*}(\alpha)\big)\wedge\pi_D^*\big((f^l)^{*} (\beta)\big)\rp \\
    & = \lp \widehat{\mathbb{T}}^{+}_{s,l}\wedge \pi_\Gr^*\big(d^{-l}(\widehat{f}^l)_{*}(\alpha)\big),\pi_D^*\big(d^{-l}(f^l)^{*} (\beta)\big) \rp.   
    \end{align*}
    
    By \eqref{eq:vanish}, when $\alpha$ is a positive $(k-p,k-p)$-form, we have $\d\big(\widehat{\mathbb{T}}^{+}_{s,l}\wedge \pi_\Gr^*(d^{-l}(\widehat{f}^l)_{*}(\alpha))\big)=0$, and {the} h-dimension of $\widehat{\mathbb{T}}^{+}_{s,l}\wedge \pi_\Gr^*\big(d^{-l}(\widehat{f}^l)_{*}(\alpha)\big)$ with respect to the projection $\pi_\Gr$ is 0. According to \cite[Lemma 3.3]{DS18JAG}, we have the decomposition
    $$ 
        \widehat{\mathbb{T}}^{+}_{s,l}\wedge\pi_\Gr^*\big(d^{-l}(\widehat{f}^l)_{*}(\alpha)\big)=\int_{y\in\Gr(D,k-p)} T_{y,l} \d\mu_l(y),
    $$ 
    where $\mu_l$ is a positive measure on $\Gr(D,k-p)$ and $T_{y,l}$ is a positive closed horizontal current of bi-dimension $(p,p)$ on the fibre $\pi_\Gr^{-1}(y)\cong D$. After multiplying {$\mu_l$ by} some positive function, we can assume $T_{y,l}\in \mathscr{C}^{1}_{h}(D)$ for all $y$ and $l$. Recall that $ \mathscr{C}^{1}_{h}(D)$ is the set of all positive closed horizontal $(k-p,k-p)$-currents of slice mass 1 on $D$. Then we get 
    $$
        \lp \widehat{\mathbb{T}}^{+}_{s,l}\wedge \pi_\Gr^*\big(d^{-l}(\widehat{f}^l)_{*}(\alpha)\big), \pi_D^*\big(d^{-l}(f^l)^{*} (\beta)\big)\rp = \int_{\Gr(D,k-p)}\lp d^{-l}(f^l)_{*}(T_{y,l}), \beta \rp \d\mu_l(y).
    $$
    
    A key observation is that mass of $\mu_l$ depends only on $\alpha$ {and is} independent of $l$. In fact, let $\beta_0$ be a smooth form in $\mathscr{C}_v^1(D)$. Then {we have} $\lp d^{-l}(f^l)_{*}(T_{y,l}), \beta_0 \rp=1$. It follows that $\|\mu_l\|=\lp \widehat{\bbT}_s^+,\pi_\Gr^*(\alpha)\wedge\pi_D^*(\beta_0)\rp$. Define a positive closed current $\widehat{T}^+$ of bi-dimension $(k-p,k-p)$ on $\Gr(D,k-p)$ by 
    $$ 
        \widehat{T}^+:=(\pi_{\Gr})_*(\widehat{\bbT}_s^+\wedge(\pi_{D})^*\beta_0).
    $$
    Then {we obtain} $\lp \widehat{T}^+,\alpha\rp=\|\mu_l\|$, which is also independent of the choice of $\beta_0$.

    Recall that $d^{-l}(f^l)_*S$ converges to $T^{-}$ uniformly on $S\in\mathscr{C}_h^1(D)$, see \cite[Theorem 4.6]{DNS}. Since $\|\mu_l\|=\lp \widehat{T}^+,\alpha\rp$ is independent of $l$ and $T_{y,l}\in \mathscr{C}^{1}_{h}(D)$, we deduce that 
    \[
        \lim_{l\to\infty} \int_{\Gr(D,k-p)}\lp d^{-l}(f^l)_{*}(T_{y,l}), \beta \rp \d\mu_l(y)=\lp \widehat{T}^+,\alpha\rp\cdot\lp T^-,\beta\rp .
    \] 
    Therefore, we have $\lp \widehat{\mathbb{T}}^{+}_{s},\pi_\Gr^*(\alpha)\wedge\pi_D^*(\beta)\rp =\lp \widehat{T}^+,\alpha\rp\cdot\lp T^{-}, \beta\rp $. This {and \eqref{eq:vanish} together imply that} $\widehat{\mathbb{T}}^{+}_{s}=\widehat{T}^{+}\otimes T^{-}$.
\end{proof}

Observe that the push-forward of $\widehat{T}^+$ to $D$ is equal to $T^+$ as the push-forward of $\widehat{\mathbb{T}}^{+}_s$ to $D\times D$ is equal to $\mathbb{T}^+$.

\begin{lemma}\label{lem:slicelift}
    Let $\widehat{T}^+$ be as in \cref{decompcurr} and $\widehat{\gamma}$ be as in \cref{exam:gamma-hat}. Then $\widehat{T}^+$ is the limit of $d^{-n_j}(\widehat{f}^{n_j})^*(\widehat{\gamma})$. In particular, it is a lift current of $T^+$. 
\end{lemma}

\begin{proof}[Sketch of the proof]
     It suffices to prove that for $\rho$-almost every $\tau$, $d^{-n_j}(\widehat{f}^{n_j})^*([\widehat{L_\tau}])\to \widehat{T}^+$ as $j\to\infty$. The last assertion is then a consequence of \cref{thm:tame}. We prove that any limit $S$ of $d^{-n_j}(\widehat{f}^{n_j})^*[\widehat{L_\tau}]$ is larger than or equal to $\widehat{T}^+$. Then, since the projections of both $S$ and $\widehat{T}^+$ to $D$ {are} equal to $T^+$, the $D$-dimension of $S-\widehat{T}^+$ is strictly smaller than $k-p$, from which it is easy to deduce $S=\widehat{T}^+$. To prove $S\geq \widehat{T}^+$, we apply the slicing theory of woven currents developed in \cite{DS16}. For details, see the proof of \cite[Lemma 5.8]{DS16}; since the argument there is local, it applies equally in the present setting.
\end{proof}

\begin{proof}[End of the proof of \cref{prop:intersec-mu-hat}]
   {The proof is essentially the same as that of \cite[Proposition 5.4]{DS16} with minor modifications adapted to the present setting.} Recall that $\Pi:\Gr(D\times D,k)\to D\times D$ is the canonical projection. By \cref{impor-h=0}, the h-dimension of $\widehat{\mathbb T}^+$ along $ \Pi^{-1}(\Delta) $ is 0. Let $\widehat{\nu}$ denote the shadow of a tangent current of $\widehat{\mathbb T}^+$ along $ \Pi^{-1}(\Delta)$. It suffices to show that $\widehat{\nu}=\widehat{\mu}^\Delta$. By \cref{SuppofT}, $\widehat{\nu}$ is a positive measure in $\Gr(D,k-p)\times \Gr(D,p)$.  We restrict $\Pi$ to $\Gr(D,k-p)\times \Gr(D,p)$ and still denote it by $\Pi$. Since $\Pi_{*}(\widehat{\mathbb T}^+)=\mathbb{T}^{+}$, we deduce that $ \Pi_{*}(\widehat{\nu})=\mu^{\Delta}$. Consider the disintegration of $\widehat{\nu}$ with respect to $\mu^\Delta$: $\widehat{\nu}=\int \widehat{\nu}_x\d\mu^\Delta$, where $\widehat{\nu}_x$ is a probability measure on $\Pi^{-1}(x,x)\cong\G(k-p,k)\times \G(p,k)$. By definition of $\widehat{\mu}^\Delta$, we also have  $\widehat{\mu}^\Delta=\int \delta_x\d\mu^\Delta$ where $\delta_x$ is the Dirac measure at the point $([E_s(x)],[E_u(x)])$. 

    Let $\Pi_{+}: \Gr(D,k-p)\times D\to D\times D$ be the projection map. \cref{thm:sp-density} implies that $\widehat{T}^+\curlywedge\pi_{D}^{*}(T^-)$ is well-defined and equal to $\widehat{T}^+\wedge\pi_{D}^{*}(T^-)$ as $T^{-}$ has continuous super-potentials on $D$. By \cref{equidensity}, the tangent current of $\widehat{\mathbb{T}}^+_s=\widehat{T}^+\otimes T^-$ along $\Pi^{-1}_{+}[\Delta]$ is {unique and given by} the pull-back of $ \widehat{T}^+\wedge\pi_{D}^{*}(T^-)$. Since $\widehat{\mathbb{T}}^{+}_{s}=(\pi_{s})_{*}(\widehat{\mathbb{T}}^{+})$, we conclude $(\pi_{s})_{*}(\widehat{\nu})=\widehat{T}^+\wedge\pi_{D}^{*}(T^-)$. By \cref{lem:slicelift}, (\ref{eq:continuous-sp}) and applying \cref{thm:unstable-Oseledec} for $f^{-1}$, we have 
    \begin{equation}\label{eq:nuhat=mu_+}
         (\pi_{s})_{*}(\widehat{\nu})=\widehat{T}^+\wedge\pi_{D}^{*}(T^-)= \lim_{j\to \infty} d^{-n_j} (\widehat{f}^{n_j})^*(\widehat{\gamma}) \wedge \pi_{D}^{*} (T^-)=\mu_+.
    \end{equation}
    Denote by $\pi_1':\G(k-p,k)\times \G(p,k)\to \G(k-p,k)$ the canonical projection. Then (\ref{eq:nuhat=mu_+}) implies $(\pi_1')_*(\widehat{\nu}_x)=\delta_{[E_s(x)]}$ for $\mu^\Delta$-almost every $(x,x)\in\Delta$. We have similar result for the projection to $\G(p,k)$. Consequently, $\widehat{\nu}_x=\delta_x$ for $\mu^\Delta$-almost every $(x,x)\in \Delta$.
\end{proof}

For each $n\geq 1$, {let $\widetilde{\Gamma}_n$ be the set of points $(x,y,v)\in \Gr(D\times D,k)$ such that $(x,y)\in \Gamma_n$ and $v\in \G(k,2k)$ is not transverse to $\Gamma_n$ at $(x,y)$.} Then $\widetilde{\Gamma}_n$ is an analytic subset of $\Gr(D\times D,k)$ of pure dimension $k^2+k-1$. The following is the key ingredient of the proof of \cref{thm:main}. 

\begin{corollary}\label{tangent-step-2}
{The limit $\widetilde{\mathbb T}^+ := \lim_{j\to\infty} d^{-n_j}[\widetilde{\Gamma}_{n_j}]$ exists,} and {the tangent current of $\widetilde{\bbT}^+$ along $\widehat{\Delta}$ is zero.} 
\end{corollary}

\begin{proof}[Proof] 
    The proof is identical to \cite[Corollary 5.9]{DS16} except that we replace $\P^k$ there by $D$ and use \cref{prop:intersec-mu-hat}.
\end{proof}

\subsection{Ramified coverings near the diagonal}
For any $\vep>0$ small enough, define
\begin{align*}
    M_\vep&=\{x_1\in M: \dist(x_1,\partial M)>\vep\};\\
    N_\vep&=\{x_2\in N: \dist(x_2,\partial N)>\vep\};\\
    D_\vep&=M_\vep\times N_\vep.
\end{align*}
Assume that $\dist(M',\partial M)\geq \vep_0$ and $\dist(N',\partial N)\geq \vep_0$ for some $\vep_0>0$. We use the coordinates $(x,z)=(x,y-x)$ and $\Delta$ can be identified with $\{z=0 \}$. Let $\pi$ denote the projection $(x,z)\mapsto z$. 

\begin{lemma}\label{lemma:vep_0/2}
    {Take} $0<\vep<\vep_0/2$. Then $\Gamma_n\cap \{\|z\|<\vep\}$ is contained in $D_{\vep_0/2}\times D_{\vep_0/2}$ for all $n\geq 1$.
\end{lemma}

\begin{proof}
    Consider any $(x,y)\in \Gamma_n\cap\{\|z\|<\vep\}$. This implies $x,y\in D'$ and $\|x-y\|<\vep$. Suppose $x=(x_1,x_2)$ and $y=(y_1,y_2)$. Then {we get} $x_1\in M'$ and $\|x_1-y_1\|<\vep$. By triangle inequality,
    \[
        \dist(y_1,\partial M)\geq \dist(x_1,\partial M)-\|x_1-y_1\|>\vep_0/2.
        \]
    On the other hand, $y_2\in N'\subset N_{\vep_0/2}$. Therefore $y\in D_{\vep_0/2}$. {A} similar proof shows that $x\in D_{\vep_0/2}$.
\end{proof}

The following proposition implies $f^n$ has exactly $d^n$ fixed points counting multiplicity. This is proven by Dujardin \cite{Dujardin04} when $k=2$.

\begin{proposition}\label{prop:ramified-covering}
     Let $\vep$ be as in \cref{lemma:vep_0/2}. Then $\pi|_{\Gamma_n\cap\{\|z\|<\vep\}}$ is a ramified covering of degree $d^n$ over $\B_{k}(0,\vep)$ for all $n\geq 1$.
\end{proposition}

\begin{proof}
    By \cref{lemma:vep_0/2}, $\pi|_{\Gamma_n\cap\{\|z\|<\vep\}}$ is a proper map to $\B_k(0,\vep)$. Let $V$ be its image onto $\B_k(0,\vep)$. Remmert's proper mapping theorem implies that $V$ is an analytic subset of $\B_k(0,\vep)$. Given any point $z\in \B_k(0,\vep)$, $\pi^{-1}(z)\cap \Gamma_{n}$ is a proper analytic subset of $\pi^{-1}(z)\cap (D\times D)$ which can be identified with an open subset of $\mathbb{C}^{k}$. This implies $\pi^{-1}(z)\cap \Gamma_{n}$ is a compact analytic subset of $\C^k$, and hence cannot contain positive-dimensional analytic components. Therefore, $\pi^{-1}(z)\cap \Gamma_{n}$ is a finite set, and we have $\dim V=\dim\Gamma_{n}=k$. Hence we must have $V= \B_k(0,\vep)$ and $\pi|_{\Gamma_n\cap\{\|z\|<\vep\}}$ is surjective. The first assertion follows from the fact that any proper surjective finite map between two complex manifolds is a ramified covering (see for example {\cite[page 179]{GR1984}}).

     {Denote by $d_n$} the degree of $\pi|_{\Gamma_n\cap\{\|z\|<\vep\}}$. Although $\Delta$ is not horizontal, by doing perturbation to $\Delta$, we can find an affine plane $L$ of dimension $k$ which is horizontal in $D\times D$ and still intersects with $\Gamma_n$ at $d_n$ points counting multiplicity. Then $[\Gamma_n]\wedge [L]$ is a measure of mass $d_n$. On the other hand, since the slice measures of both $[L]$ and $d^{-n}[\Gamma_n]$ are 1, $d^{-n}[\Gamma_n]\wedge [L]$ is a probability measure (see the discussion before \cref{example:ver-strict-posi}). It follows that $d_n=d^n$.
\end{proof}

We are interested in the number of graphs over $\B_k(0,\vep)$ with respect to $\pi|_{\Gamma_n}$. More precisely, they are graphs of holomorphic maps from {$\B_k(0,\vep)$} to $\C^k$ which are contained in $\Gamma_n$. By Cauchy's formula, the angle between such a graph and $\Delta$ is bounded from below by some constant independent of the graph and $n$. We will show that the number of graphs is close to $d^n$ when $n$ is large; {see \cref{cor:lotsgraphs} for the precise statement.} This will give us enough transversality between $\Gamma_n$ and $\Delta$ to show equality (\ref{commu-intersec}). This number is closely related to the ramification of $\Gamma_n$ near $\Delta=\pi^{-1}(0)$. For instance, when $\pi|_{\Gamma_n\cap\{\|z\|<\vep\}}$ has no ramification, we have exactly $d^n$ graphs. 

We consider a more general setting: let $\Gamma$ be a $k$-dimensional submanifold of $ \B_{k}(0,3)\times \B_{l}(0,2)$, which is contained in $\B_{k}(0,3)\times \B_{l}(0,1)$. Let $\pi_0: \B_{k}(0,3)\times \B_{l}(0,2)\to \B_{k}(0,3)$ be the canonical projection. As in \cref{prop:ramified-covering}, we can show $\pi_0|_{\Gamma}$ is a surjective ramified covering from $\Gamma$ to $\B_{k}(0,3)$ and suppose its degree is $d$. 

Note that $L_0:=\pi_0^{-1}(0)$ is an $l$-dimensional subspace of $\C^{k+l}$ and thus defines an element $[L_0]$ in $\G(l,k+l)$. For each complex $k\times l$-matrix $A$ whose coefficients have modulus smaller than $1$, it defines an element in $\G(l,k+l)$ by the equation $z'=A z''$  with $z=(z',z'')\in \C^k\times\C^l$. For simplicity, we identify $A$ with  this element. Then $\{A:\|A\|<\vep\}$ is an open neighborhood of $[L_0]$ in $\G(l,k+l)$ for any $\vep>0$, and $A$ also serves as the local coordinate. Let $\widehat{L}_0$ be the canonical lift of $L_0$ to $\Gr(\B_{k}(0,3)\times \B_{l}(0,2),l)=\B_{k}(0,3)\times \B_{l}(0,2)\times \G(l,k+l)$. Then {we have} $\widehat{L}_0=\{(z,[L_0]):z\in L_0\}$. {Define} $V_\vep:=\B_k(0,\vep)\times \B_l(0,2)\times \{ \|A\|<\vep\}$. This is an open neighborhood of $\widehat{L}_0$.

Define $\widetilde{\Gamma}$ as in \cref{tangent-step-2}, which is an analytic set of pure dimension $(kl+k-1)$ in $\B_k(0,3)\times \B_l(0,2)\times \G(l,k+l)$. Let $\omega_{\C^k}$ and $\omega_{\G}$ denote the standard K\"ahler forms on $\C^{k}$ and the unitary K\"ahler form on $\G(l,k+l)$, respectively. Define 
$$ 
    \Omega_{k,l}:=\pi_{\C^k}^*(\omega_{\C^k}^{k-1})\wedge \pi_\G^*(\omega_\G^{kl}),
$$ 
{which is} a smooth form of bi-degree $(kl+k-1,kl+k-1)$ on $\B_k(0,3)\times \B_l(0,2)\times \G(l,k+l)$ where $\pi_{\C^k}$ and $\pi_\G$ are the obvious projection maps. The following result quantifies the number of inverse graphs of $\pi_0|_{\Gamma}$ when restricted {to} a small neighbourhood of $0 $ in $ \B_k(0,3)$. For the proof, see \cite{D05JGA,DS16}.

\begin{theorem}[\cite{DS16}, Proposition 3.16]\label{dynadesigraphs}
Fix $\delta>0$ and $0<\vep<1$. Let $\Gamma$, $\widetilde{\Gamma}$, and $V_\vep$ be as above. {Define} $\widetilde{\Gamma}_\vep:=[\widetilde{\Gamma}\cap V_\vep]$. Then $\Gamma$ contains at least
$$
    (1-2\delta)d-C\delta^{-1}\vep^{-2(kl+k-1)} \lp \widetilde{\Gamma}_\vep,\Omega_{k,l} \rp
$$ 
graphs over $\B_k(0, c_0\vep)$ along the $\pi_0$-direction. Here, $C>0$ and $0<c_0<1$ are two constants depending only on $k$ and $l$.
\end{theorem}

We apply \cref{dynadesigraphs} to the case when $\Gamma=\Gamma_n$ and $L_0=\Delta$. It follows that $\Gamma_n$ is mostly transverse to $\Delta$ when $n\to \infty$. 

\begin{corollary}\label{cor:lotsgraphs}
    For any $\delta>0$, there exists $\vep>0$ such that for any subsequence $\{n_j\}$, we can extract a sub-subsequence $\{n_{j_l}\}$ such that when $l$ is large enough,  $\Gamma_{n_{j_l}}$ has at least $(1-3\delta)d^{n_{j_l}}$ graphs over $\{\|z\|<\vep\}$ along the diagonal direction.
\end{corollary}

\begin{proof}
   Let $\pi_{\G}$ be the canonical projection from $\Gr(D \times D,k)$ to $\G(k,2k)$ and $\omega_{\G}$ be the unitary K\"ahler form on $\G(k,2k)$. Recall that $\Pi$ is the canonical projection map from $\Gr(D\times D,k)$ to $D\times D$. We apply \cref{dynadesigraphs} to our setting where $k=l$ and $\pi=\pi_0$ is the projection map from $\C^k \times \C^k$ to $\C^k$ along the diagonal direction. 
   
   For any subsequence $\{n_j\}$, by \cref{tangent-step-2}, we can always find a subsequence $\{n_{j_l}\}$ such that $ d^{-n_{j_{l}}}[\widetilde{\Gamma}_{n_{j_l}}]$ converges to some $\widetilde{\mathbb{T}}^{+}$. We can identify $\Gr(D\times D,k)$ with an open subset of the normal bundle of $\widehat{\Delta}$. Recall that tangent currents are obtained as limit values of $(A_{\lambda})_{*}(\widetilde{\bbT}^{+})$ {as} $\lambda \to \infty$, where $A_{\lambda}$ is the dilation along the normal bundle of $\widehat{\Delta}$. Notice that $A_{1/\vep}$ maps $V_\vep$ to $V_1$. Therefore, {we have}
   \[
       \lp (A_{1/\vep})_*(\widetilde{\bbT}^+)|_{V_1}, \Omega_{k,k}\rp=\lp (A_{1/\vep})_*(\widetilde{\bbT}^+|_{V_\vep}), \Omega_{k,k}\rp=\lp \widetilde{\bbT}^+|_{V_\vep}, A_{1/\vep}^*(\Omega_{k,k})\rp.
   \]
   {For the first term above}, since the tangent current of $\widetilde{\bbT}^+$ along $\widehat{\Delta}$ is 0, by definition we have 
   \[
       \lim_{\vep\to 0}\lp (A_{1/\vep})_*(\widetilde{\bbT}^+)|_{V_1}, \Omega_{k,k}\rp=0.
   \]
   {For the last term}, we have $A_{1/\vep}^*(\Omega_{k,k})=\vep^{-2(k^2+k-1)}\Omega_{k,k}$ in a neighborhood of $\widehat{\Delta}$. Therefore, for any $\delta>0$, there exists $\vep>0$ such that 
   \[
       \lp \widetilde{\bbT}^+|_{V_\vep}, \Omega_{k,k}\rp\leq 2C^{-1}\delta^2\vep^{2(k^2+k-1)}
   \]
   where $C>0$ is the constant in \cref{dynadesigraphs}. Since  $ d^{-n_{j_{l}}}[\widetilde{\Gamma}_{n_{j_l}}]$ converges to $\widetilde{\mathbb{T}}^{+}$, when $l$ is large enough, we have $\lp[\widetilde{\Gamma}_{n_{j_l}}\cap V_\vep], \Omega_{k,k} \rp<C^{-1}\delta^2\vep^{2(k^2+k-1)}d^{n_{j_l}}$. By \cref{dynadesigraphs}, this implies ${\Gamma}_{n_{j_l}}$ has at least $(1-2\delta)d^{n_{j_l}}-\delta d^{n_{j_l}}=(1-3\delta)d^{n_{j_l}}$ graphs over $ \B_{k}(0,c_0\vep)$ along the diagonal direction.  
\end{proof}

Before we proceed, we define another type of dynamical degrees which are introduced in \cite{BDR23}. For $0\leq s<p$, we define 
  $$
   \lambda^+_{s,n}:=\int_{\Gamma'_n} \pi_1^{*}(\omega^{k-s})\wedge \pi_2^{*}(\omega^{s}),
  $$
  and for $0\leq s <k-p$
  $$
  \lambda^-_{s,n}:=\int_{\Gamma'_n} \pi_1^{*}(\omega^{s})\wedge \pi_2^{*}(\omega^{k-s}),
  $$ 
  where $\omega$ is the standard K\"ahler form on $D$ and $\Gamma_n'$ is the restriction of $\Gamma_n$ on $D_{\vep_0/2}\times D_{\vep_0/2}$. Define
  \[
      \lambda_s^+:=\limsup_{n\to\infty}(\lambda_{s,n}^+)^{1/n} \quad \text{and} \quad \lambda_s^-:=\limsup_{n\to\infty}(\lambda_{s,n}^-)^{1/n}.
  \]
  It is easy to see that $\lambda_s^\pm\leq d_s^\pm$ for all possible $s$. 
  \medskip

  Let $a$ be a fixed point of $f^n$. We say {that} $a$ is a {\it saddle} periodic point if none of the eigenvalues of $\d f^n(a)$ have modulus 1. In fact, each of these saddle points has exactly $p$ eigenvalues with modulus $>1$ and $k-p$ eigenvalues with modulus $<1$. For a graph obtained in \cref{cor:lotsgraphs}, we say {that} it is {\it saddle} if its intersection with $\Delta$ is a saddle periodic point. The following result says that most of these graphs are saddle.

\begin{corollary}\label{saddlegraphs}
    Let $\delta $, $ \vep$, and $n_{j_l}$ be as in \cref{cor:lotsgraphs}. When $l$ is large enough, $\Gamma_{n_{j_l}}$ has at least $(1-5\delta)d^{n_{j_l}}$ saddle graphs over $\{\|z\|<\vep\}$.
\end{corollary}

\begin{proof} [Proof]
   {According to \massump, there exist $0<\theta<1$ and $C_0>0$ such that $\lambda^{+}_{s,n}d^{-n} < C_0\theta^n$ and $\lambda^{-}_{s,n}d^{-n} < C_0\theta^n$ for every $s$ and $n$. Denote by $\{\gamma_l^m\}_{m=1}^{k_l}$ the graphs over $\{\|z\|\leq \vep\}$ obtained from $\Gamma_{n_{j_l}}$ in \cref{cor:lotsgraphs}. Then we have $k_l\geq (1-3\delta)d^{n_{j_l}}$. For a constant $C>0$ and $s\neq p$, denote by $\mathcal{F}_{l,C}^s$ the set of indices $m$ such that the following is true:
  \[
      \int_{\gamma_l^m} \pi_1^{*}(\omega^{k-s})\wedge \pi_2^{*}(\omega^{s})\leq C\theta^{n_{j_l}}.
  \]

  Denote by $\mathcal{F}_{l,C}:=\bigcap_{s\neq p}\mathcal{F}_{l,C}^s$. It is enough to prove the following claim.
  
  \medskip
\noindent \textbf{\textit{Claim:}} there exists a constant $C>0$ such that $\#\mathcal{F}_{l,C}\geq (1-5\delta)d^{n_{j_l}}$ for all $l$.
\medskip

  In \cite[Proposition 5.13]{DS16}, this is obtained by a cohomological argument. Since that argument is not available in our setting, we prove the required estimate directly. Once this claim is established, the remaining part of the proof of \cite[Proposition 5.13]{DS16} can be applied. Indeed, the argument there is local in nature and therefore applies without change in the present setting. 

    We proceed to prove the claim. By the definitions of $\lambda_{s,n}^\pm$, for all $l$ and $s\neq p$, we have 
    \begin{align*}
        C_0d^{n_{j_l}}\theta^{n_{j_l}}&>\int_{\Gamma'_{n_{j_l}}} \pi_1^{*}(\omega^{k-s})\wedge \pi_2^{*}(\omega^{s})\\
        &\geq \sum_{m\not\in \mathcal{F}_{l,C}^s} \int_{\gamma_l^m} \pi_1^{*}(\omega^{k-s})\wedge \pi_2^{*}(\omega^{s}) \geq C(k_l-\#\mathcal{F}_{l,C}^s)\theta^{n_{j_l}}.
    \end{align*}
    We deduce that $\#\mathcal{F}_{l,C}^s> k_l-C_0C^{-1}d^{n_{j_l}}$ for all $s\neq p$ and all $l$. Choosing a large enough $C>0$, the result follows from the facts that $k_l\geq (1-3\delta)d^{n_{j_l}}$ and $\mathcal{F}_{l,C}=\bigcap_{s\neq p}\mathcal{F}_{l,C}^s$. 
    }
\end{proof}

\begin{proof}[Proof of \cref{thm:main}]
    Define 
    \[
        \mu_{n}^{\Delta}:=d^{-n}\sum\limits_{a\in Q_{n}} \delta_{(a,a)}.
    \]
    It suffices to prove that {we have} $\mu_n^\Delta\to \mu^\Delta$ as $n\to\infty$. Let $\nu^\Delta$ be a limit of $\{\mu_n^\Delta\}$, say $\nu^\Delta=\lim_{j\to\infty} \mu_{n_j}^\Delta$. By \cref{saddlegraphs}, for any $\delta>0$, there is some $\vep>0$ and a sub-subsequence, which we still denote by $\{n_j\}$, such that $\Gamma_{n_j}$ has at least $(1-5\delta)d^{n_j}$ saddle graphs over $\{\|z\|<\vep\}$ for every $j$. These graphs are denoted by $\{\Gamma_{n_j}^l\}_l$. For each $j\geq 1$, define
    \[
        \mathbb T^{{\rm{graph}}}_j:=d^{-n_j}\sum_l [\Gamma_{n_j}^l] \quad \text{and} \quad \nu_j^\Delta:=\mathbb T^{{\rm{graph}}}_j\wedge [\Delta].
    \]
    Then $\nu_j^\Delta$ is a positive measure of mass at least $1-5\delta$. By definition, we also have $\nu_j^\Delta\leq\mu_{n_j}^\Delta$ and $\mathbb T^{{\rm{graph}}}_j\leq d^{-n_j}[\Gamma_{n_j}]$. By extracting a subsequence, we may assume $\bbT^{{\rm{graph}}}_j$ converges to some $\mathbb S\leq \bbT^+$. Suppose $\Gamma_{n_j}^l$ is the graph of $\tau_{n_j}^l$ {which is a} holomorphic map from $\{\|z\|<\vep\}$ to $\C^k$. Since ${\|\tau_{n_j}^l\|}$ is uniformly bounded for all $j$ and $l$, ${\|\tau_{n_j}^l\|_{\mathcal C^1}}$ is also uniformly bounded by Cauchy's formula. Then it is easy to prove that $\nu_j^\Delta$ also converges to a positive measure $\nu_\delta$. We have $\nu^\Delta\geq \nu_\delta$, $\mu^\Delta\geq \nu_\delta$, and $\|\nu_\delta\|\geq 1-5\delta$. Since $\delta>0$ is arbitrary, we must have $\nu^\Delta=\mu^\Delta$.
\end{proof}

\bibliographystyle{alphaurl}  
\bibliography{references}  

\end{document}